\documentclass[12pt]{amsart}
\usepackage[utf8]{inputenc}
\usepackage[colorlinks=true]{hyperref}
\usepackage{amsmath,amscd,amsbsy,amssymb,amsthm,amsfonts}
\usepackage{latexsym,url,bm}
\usepackage{cite,enumitem}
\usepackage{ulem,xcolor}
\usepackage{tikz}
\usepackage{fullpage}

\newtheorem{theorem}{Theorem}[section]
\newtheorem{lemma}[theorem]{Lemma}
\newtheorem{proposition}[theorem]{Proposition}
\newtheorem{corollary}[theorem]{Corollary}

\numberwithin{equation}{section}
\numberwithin{figure}{section}

\title{\textbf{On the $L^2$ well-posedness and decay estimate of third order Benjamin-Ono equation}}
\author{Lizhe Wan}
\address{Department of Mathematics, University of Wisconsin - Madison}
\curraddr{}
\email{lwan33@wisc.edu}
\thanks{}
\begin{document}

\begin{abstract}
We consider the $L^2$ well-posedness of  third order Benjamin-Ono equation. 
We show that by means of a normal form and a gauge transformation, the equation can be changed into an Airy-type equation. 
A second goal of this work is to establish that the solutions to the nonlinear  third order Benjamin-Ono equation problem exhibit a dispersive decay estimate analogue to the corresponding linear associated problem. The key ingredient is the use of a nonlinear vector field akin to the work in \cite{ifrim2019dispersive, MR3948114}.
\end{abstract}
\maketitle
\tableofcontents

\section{Introduction}
In this article we consider third order Benjamin-Ono equation
\begin{equation}
 \phi_t - \phi_{xxx}+ \frac{3}{4}\phi^2\phi_x+\frac{3}{4}[\phi H\phi_x+H(\phi \phi_x)]_x=0, \qquad \phi(0) = \phi_0 .\label{HBO}
\end{equation}
Here $\phi$ is a real valued function: $\mathbb{R}\times\mathbb{R}\rightarrow \mathbb{R}$, and $H$  is the Hilbert transform on the real line, which is given via a Fourier multiplier with symbol $-i$ sgn$\xi$.
\par As its name suggests, third order Benjamin-Ono equation derives from the Benjamin-Ono equation, and it is part of the so called \textit{Benjamin-Ono hierarchy}
\begin{equation}
  \phi_t + H \phi_{xx} = \phi \phi_x, \qquad \phi(0) = \phi_0. \label{BenjaminOno}
\end{equation}
It was introduced by Benjamin \cite{benjamin1967internal} and Ono\cite{MR398275} in order to describe long internal waves in a two-layer fluid of great depth.
The Benjamin-Ono equation is completely integrable, and, hence possesses many conserved quantities associated to smooth solutions. Here we list first three conserved quantities: 
\begin{align*}
    E_0 &= \int \phi^2 \, dx, \\
    E_1 &= \int \phi H \phi_x- \frac{1}{3}\phi^3 \, dx,\\
    E_2 &= \int \phi_x^2 - \frac{3}{4}\phi^2H \phi_x + \frac{1}{8}\phi^4\, dx.
\end{align*}
More generally, at each nonnegative integer $k$, there exists a conserved energy $E_k$ corresponding at leading order to $\dot{H}^{\frac{k}{2}}$ norm of $\phi$ (see \cite{MR1607987} and \cite{MR550203}).
\par The Benjamin-Ono equation itself is generated by the Hamiltonian $E_1$,
and the symplectic form
\begin{equation*}
    \omega(\psi_1, \psi_2) = \int \psi_1 \partial_x \psi_2\, dx,
\end{equation*}
with the associated map $J = \partial_x$.
In the same fashion, the Hamiltonian $E_2$ and the symplectic form $\omega$ generate the new equation \eqref{HBO}, which is called \textit{third order Benjamin-Ono equation.} 
See \cite{MR3972074} for a detailed survey of the Benjamin-Ono equation and its hierarchy.

\par Equation \eqref{HBO} is also a completely integrable dispersive equation. The group velocity of waves depends on the frequency, which is similar to KdV equation. 
Above conserved quantities such as $E_0$ and $E_2$ of Benjamin-Ono equation are also the conserved for \eqref{HBO}. 
Tools for dispersive partial differential equations like Strichartz estimates can be used for \eqref{HBO}. 
It also shares a lot of similarities with the Benjamin-Ono equation \eqref{BenjaminOno}.
If $\phi(t,x)$ is a solution, so is $\lambda\phi(\lambda^3 t, \lambda x)$, which suggests that the scale invariant Sobolev space associated to this scaling is $\dot{H}^{-\frac{1}{2}}(\mathbb{R})$. 
This scale invariant space is the same as the Benjamin-Ono equation but different from KdV equation.
\par Our first goal of this article is to prove the following result of \eqref{HBO}.
\begin{theorem}
The third order Benjamin-Ono equation \eqref{HBO} is locally well-posed in $L^2(\mathbb{R})$. \label{t:MainTheoremOne}
\end{theorem}

Here by locally well-posedness we mean the unique existence of solution and continuous dependence of initial data; this is nothing more than the Hadamard local well-posedness theory. 
Since the $L^2$ norm of the solution is conserved, this also implies  global well-posedness of the equation \eqref{HBO}.
\par The second goal here is to obtain a decay estimate for the positive time $t$. 
\begin{theorem}
Assume the initial data $\phi_0$ for the equation $\eqref{HBO}$ is small in the sense that for any small constant $\delta>0$.
\begin{equation}
    \| \phi_0\|_{\dot{B}^{-\frac{1}{2}+\delta}_{2,\infty}} + \|x \phi_0 \|_{\dot{H}^{\frac{1}{2}}}\leq \epsilon \ll 1 ,\label{SmallnessCondition}
\end{equation}
then for the lifespan $t \ll \epsilon^{-12}$, we  have the dispersive bound for $\phi$:
\begin{equation*}
    t^{\frac{1}{4}}\left \langle x \right \rangle_{t}^{\frac{1}{4}-\delta}|\phi(t,x)|+ t^{\frac{3}{4}}\left\langle x \right \rangle_{t}^{-\frac{1}{4}-\delta}|\phi_x(t,x)|\lesssim \epsilon.
\end{equation*}
Furthermore, in the elliptic region $E = \{x\lesssim -t^{\frac{1}{3}} \}$ we have a better bound
\begin{equation*}
   \left\langle x \right \rangle_{t} |\phi(t,x)| + t^{\frac{1}{2}}\left\langle x \right \rangle_{t}^{\frac{1}{2}}|\phi_x(t,x)|\lesssim \epsilon \log (t^{-\frac{1}{3}}\left\langle x \right \rangle_{t}).
\end{equation*}
Here, $\left\langle x \right \rangle_{t}$ denotes the time-depended Japanese bracket given by $\left\langle x \right \rangle_{t} := (x^2+ t^{\frac{2}{3}})^{\frac{1}{2}}$, which is consistent with Airy scaling. \label{t:MainTheoremTwo}
\end{theorem}
 This result is similar with the one in \cite{ifrim2019dispersive} but in our case the time for which the nonlinear solution disperses as the linear solution is longer, i.e.. $t \ll \epsilon ^{-12}$.

Historically, in \cite{MR1396715}, Feng and Han proved that \eqref{HBO} is locally well-posed in $H^s(\mathbb{R}), s\geq 4$.  
Feng further established the well-posedness of \eqref{HBO} in weighted Sobolev space $H^s_\gamma, s\geq 3, \gamma\in [0,1]$, see \cite{MR1476042}.
Linares, Pilod and Ponce in their paper \cite{MR2737850} considered local well-posedness in $H^s(\mathbb{R}), s\geq 2$ for more general third order type
Benjamin-Ono equations.
Molinet and Pilod in
\cite{MR3005535}  proved global well-posedness in $H^s(\mathbb{R}), s\geq 2$. 
There are also results for \eqref{HBO} on the torus.  Tanaka \cite{MR3992041} showed that a  more general third order type Benjamin-Ono type of equations is well-posed in $H^s(\mathbb{T}), s\geq \frac{5}{2}$. 
Gassot \cite{MR4227053} proved that for any $t\in \mathbb{R}$, the third order Benjamin-Ono flow map continuously extends to $H^s_{r,0}(\mathbb{T})$ if $s\geq 0$, but does not admit a continuous extension to $H^{-s}_{r,0}(\mathbb{T})$ if $0<s< \frac{1}{2}$.
More generally, people even consider higher order Benjamin-Ono equation, such as the fourth order Benjamin-Ono equations, see for instance the work of Tanaka in \cite{MR4199802} for both torus and the real line. 

Our goal in this present work is to provide a clean $L^2$ local well-posedness result for the Cauchy problem \eqref{HBO} on the real line, improving considerably the previous work accomplished by the above listed authors in the analysis of the same problem. We use novel ideas introduced by Ifrim-Tataru, and Ifrim-Koch-Tataru in \cite{MR3948114} and \cite{ifrim2019dispersive} in order to achieve both our results. The difficulties we encountered are however different in nature than in the work cited above, and hence require an adapted methodology. 
 
\par The article is organized as follows.
In Section~\ref{s:Definition}, we first give the definitions and notations that will be needed in the later proofs.
Then we establish the Strichartz and bilinear Strichartz estimates for the linear Airy flow.

In Section~\ref{s:NormalForm}, we perform a partial  normal form method  transformation first introduced by Simon in \cite{MR719852} followed by Shatah \cite{MR803256}.  
The main principle in the normal form method is to apply a quadratic correction to the unknown in order to replace a nonresonant quadratic nonlinearity by a milder cubic nonlinearity.
 Unfortunately this method does not apply directly here, because some terms
in the quadratic correction are unbounded, and so are some of the cubic terms generated by
the correction. 
To bypass this issue, we apply the techniques developed in \cite{ifrim2019dispersive}, which is carried out in two steps:
\begin{enumerate}
\item a partial normal form transformation which is bounded and removes some of the
quadratic nonlinearity,
\item a gauge transform which removes the remaining part of the quadratic nonlinearity in a bounded way.
\end{enumerate}
This will transform \eqref{HBO} into an equation where the quadratic terms have been removed and replaced by cubic perturbative terms.
We then finalize the proof of apriori $L^2$ bound for \eqref{HBO} using a bootstrap argument.

In Section~\ref{s:LinearizedHBO} we show that the linearized third order Benjamin-Ono equation is
well-posed in $L^2$.
In the proof we also use a normal form transformation to eliminate quadratic terms on the right, and replace them by cubic terms. 
The difference with respect to the computation in Section~\ref{s:NormalForm} is that here we leave certain quadratic terms on the right, because their corresponding normal form correction would be too singular.
We proceed to establish the estimates  using a bootstrap argument .

As a consequence of well-posedness of the linearized equation, we get the weak Lipschitz dependence on the initial data for \eqref{HBO}.
In Section~\ref{s:WellPosedness} we use this important property to establish Theorem~\ref{t:MainTheoremOne}.
\par Section~\ref{s:Decay} is devoted to the sketch of Theorem~\ref{t:MainTheoremTwo}, where we outline three main steps in  the same spirit as in the work of \cite{ifrim2019dispersive}.
Here the estimate differs from the corresponding result of \cite{ifrim2019dispersive} by a $\delta$ exponent.
This is because in the result of Proposition \ref{t:BOReference}, $\phi$ cannot reach $B^{-\frac{1}{2}}_{2,\infty}$ regularity.
In Section~\ref{s:AlmostConserved}, we prove that the $\dot{H}^{-\frac{1}{2}}$ norm of the linearized solution $v$ is almost conserved, in which the lifespan of the solution $\epsilon^{-12}$ is much more better than the corresponding lifespan in \cite{ifrim2019dispersive}.
In Section~\ref{s:nonlinear vf} we consider the inequality in three regions, namely the hyperbolic, self-similar and elliptic regions. 
We establish the nonlinear vector field bound in all three regions, which concludes the proof of Theorem~\ref{t:MainTheoremTwo}.

\textbf{Acknowledgments.} The author would like to thank his advisor, Professor Mihaela Ifrim, for introducing him to this research area and for many helpful discussions that lead to this result.

\section{Notations and the linear flow}\label{s:Definition}
\textit{The big $O$ notation. } We use the notation $A\lesssim B$ or $A = O(B)$ to denote the estimate that $|A| \leq CB$, where $C$ is a universal constant. 
If X is a Banach space, we use $O_X(B)$ to denote any element in X with norm $O(B)$; explicitly we say $u = O_X(B)$ if $\|u\|_X \leq CB$. \\

\textit{Hilbert transform.}
The Hilbert transform of a function $f$ on the real line is defined by
\begin{equation*}
    Hf(x) = PV \frac{1}{\pi}\int \frac{f(y)}{x-y}dy,
\end{equation*}
where $PV$ denotes the principle value.
On the Fourier side, an alternative definition of the Hilbert transform is given by
\begin{equation*}
    \widehat{Hf}(\xi) = -i\mbox{sgn}(\xi)\hat{f}(\xi).
\end{equation*}
Here we list some properties of Hilbert transform.
\begin{itemize}
\item[1.]Skew-adjointness:  
\begin{equation}
  \int uH(v) \, dx = -\int vH(u) \, dx, \mbox{ if } u,v \in L^2.  \label{HilbertOne}
\end{equation}
\item[2.] Skew identity property:
\begin{equation}
\int H(u)H(v)\, dx = \int uv \,dx, \mbox{ if } u,v \in L^2. \label{HilbertTwo}
\end{equation}
\item[3.] Convolution Identity: 
\begin{equation}
 H\left((uH(v)+vH(u)\right) = H(u)H(v) - uv, \mbox{ if } u,v \in L^2.    \label{HilbertThree}
\end{equation}

\end{itemize}
An immediate consequence of above properties is that the solution of \eqref{HBO} satisfies the $L^2$ conservation law enumerated earlier.\\

\textit{Littlewood-Paley decomposition.} 
We begin with the Riesz decomposition
\begin{equation*}
    \mathbf{1} = P_{-}+P_{+},
\end{equation*}
where
\begin{equation*}
    \widehat{P_{-}f(\xi)}: = \mathbf{1}_{(-\infty,0)}\hat{f}(\xi), \quad \widehat{P_{+}f(\xi)}: = \mathbf{1}_{[0,\infty)}\hat{f}(\xi).
\end{equation*}
Consider a bump function which is supported on $[-2, 2]$ and equal to 1 on $[-1,1]$, we define the Littlewood-Paley operator $P_{\leq k} = P_{<k+1}$  by
\begin{equation*}
    \widehat{P_{\leq k}f(\xi)}: = \psi(\frac{\xi}{2^k})\hat{f}(\xi)
\end{equation*} 
for $k\geq 0$,
and $P_{>k} : = P_{\geq k-1}:= 1 -P_{\leq k}$.
We also define $P_{k} : = P_{\geq k} -P_{\leq k-1}$.
Note that due to the Minkowski inequality, all the operators $P_k$, $P_{\leq k}$ are bounded on all translation-invariant Banach spaces. 

\par An important feature of the Littlewood-Paley operator is that $P_{\pm}$ commutes with $P_k, P_{<k}$. 
Therefore, for simplicity, we will let $P_k^{\pm} = P_kP_{\pm}$, respective $P_{<k}^\pm := P_{\pm}P_{<k}$.
We also introduce the notations $\phi_k^+: = P_k^+\phi$, and $\phi_k^- = P_k^- \phi$, respectively.
\par Given the Littlewood-Paley operator $P_k$, we further introduce the enlarged projection operator  $\tilde{P}_k$. 
On the frequency side, it equals $1$ in the support of $P_k$, and is supported on a sightly larger set than the support of $P_k$(say $2^{k-2}$).

\par From Plancherel theorem we have the bound
for inhomogeneous Sobolev and Besov spaces
\begin{equation*}
    \|f\|_{H^s_x} \approx \left(\sum_{k =0}^\infty\|P_k f\|^2_{H^s_x}\right)^{\frac{1}{2}}\approx \left(\sum_{k =0}^\infty2^{ks}\|P_k f\|^2_{L^2_x}\right)^{\frac{1}{2}}, \quad \|f\|_{B^s_{2,\infty}} = \sup_k 2^{sk} \|P_k f \|_{L^2},
\end{equation*}
for any $s\in \mathbb{R}$. 
Similar Littlewood-Paley characterizations also hold for homogeneous Sobolev and Besov spaces. 
 
 \par A very useful result of Littlewood-Paley decomposition is the Bernstein's inequality, whose proof can be found on textbook for example \cite{MR2768550}.
 
 \begin{lemma}[Bernstein's inequality]
 For $1\leq p \leq q\leq \infty$ and a function $f$ on $\mathbb{R}^d$, we have
 \begin{align*}
     \|P_k f \|_{L^q(\mathbb{R}^d)} \lesssim_{p,q,d} 2^{k\left(\frac{d}{p}-\frac{d}{q}\right)}\|P_k f \|_{L^p(\mathbb{R}^d)},\\
      \|P_{\leq k} f \|_{L^q(\mathbb{R}^d)} \lesssim_{p,q,d} 2^{k\left(\frac{d}{p}-\frac{d}{q}\right)}\|P_{\leq k} f \|_{L^p(\mathbb{R}^d)}.
 \end{align*}
 \end{lemma}

\textit{Frequency envelopes.}
We say that a sequence $\{ c_k\}_0^\infty \in l^2$ is an $L^2$ frequency envelope for $\phi \in L^2$ if 
\begin{enumerate}
    \item $\sum_{k = 0}^\infty c_k^2 \lesssim 1$ ;
    \item it is slow varying, $c_j/ c_k \leq 2^{\delta|j-k|}$, with $\delta$ a very small universal constant;
    \item it bounds the dyadic norms of $\phi$, namely $\| P_k\phi\|_{L^2}\leq c_k$.
\end{enumerate}
The idea of frequency envelopes dates back from T. Tao\cite{MR2052470}, and is later improved and implemented by D. Tataru \cite{MR1827277}.
Given a frequency envelope $c_k$ we define 
\begin{equation*}
    c_{\leq k} = \left(\sum_{j\leq k}c^2_j\right)^{\frac{1}{2}}, \qquad c_{\geq k} = \left(\sum_{j\geq k}c^2_j\right)^{\frac{1}{2}}.
\end{equation*}
We harmlessly assume that $c_0 \approx 1$. 
Similarly, for $H^s$ functions, we replace the last condition by $\| P_k\phi\|_{L^2}\leq 2^{-ks}c_k$.
\par Later for example in \eqref{InitialCondition} or \eqref{InitialOneHalf}, we consider the small initial data satisfying $\| \phi_0\|_{L^2}\leq \epsilon$.
Using the language of frequency envelopes, we can replace this condition by $\| P_k \phi\|_{L^2}\leq \epsilon c_k$ for each $k\in \mathbb{N}$.
 \\

\textit{Multi-linear expressions.} By $L(\phi_1, \dots , \phi_n)$ we denote a translation-invariant expression of the form
\begin{equation*}
    L(\phi_1, \dots , \phi_n)(x) = \int K(y) \phi_1(x+y_1)\dots\phi_n(x+y_n)\, dy,
\end{equation*}
where $K\in L^1$.
By $L_k$ we denote such multilinear expressions where output is localized at frequency $2^k$. 
An obvious observation is that the multilinear expressions behave well with respect to reiteration, e.g.
\begin{equation*}
    L(L(u,v),w) = L(u,v,w).
\end{equation*}
Multilinear $L$ type expressions can easily be estimated in terms of linear bounds for their entries.
For bilinear $L$ type expressions, we have the H{\"o}lder's type of inequality
\begin{equation*}
    \|L(u_1, u_2) \|_{L^r} \lesssim \| u_1\|_{L^{p_1}}\| u_1\|_{L^{p_2}},   \qquad \frac{1}{p_1}+ \frac{1}{p_2} = \frac{1}{r}.
\end{equation*}
For trilinear $L$ type expressions, the effect of uncorrelated translations may be taken into considerations. 
We define the translation group $\{T_y \}_{y\in \mathbb{R}}$,
\begin{equation*}
    (T_y u)(x) = u(x+y).
\end{equation*}
Then we can also estimate the trilinear form by
\begin{equation*}
\| L(u_1, u_2, u_3)\|_{L^r} \lesssim \|u_1 \|_{L^{p_1}}\sup_{y\in \mathbb{R}}\| u_2 T_y u_3\|_{L^{p_2}}, \qquad \frac{1}{p_1}+ \frac{1}{p_2} = \frac{1}{r}.
\end{equation*}
Or without ambiguity  in a simpler form
\begin{equation*}
    \| L(u_1, u_2, u_3)\|_{L^r} \lesssim \|u_1 \|_{L^{p_1}} \|L(u_2, u_3) \|_{L^{p_2}}.
\end{equation*}
When classifying cubic and higher order terms obtained after implementing a normal form
transformation, we observe that having a commutator structure can simplify our computation
We will repeatedly use the following lemma in normal form analysis for the Littlewood-Paley projection $P_k$:
\begin{lemma}[Leibniz rule for $P_k$, \cite{MR1869874}] 
\begin{equation}
    [P_k , f]g = L( f_x, 2^{-k}g). \label{commutator}
\end{equation}
\end{lemma}
This lemma can be understood in the following sense. Suppose one function $g$ has high frequency $\sim 2^k$ and the other function $f$ has low frequency $\leq 2^k$. 
Then $P_k(fg)- fP_k g$ shifts a derivative from the high-frequency function $g$ to the low-frequency function $f$.
This shift  generally ensures that all such commutator terms can be easily estimated.\\

\textit{The Linear flow.}
The linear part of \eqref{HBO} is an Airy-type equation, 
\begin{equation*}
    \left\{
             \begin{array}{lr}
            \phi_t -\phi_{xxx} = 0  &\\
             \phi(0) = \phi_0.
             \end{array}
\right.
\end{equation*}
To better understand the positive long time behavior of the solution, it is useful to separate the domain of evolution $(t,x)\in \mathbb{R}^{+}\times \mathbb{R}$ into three regions
\begin{enumerate}
    \item The hyperbolic region
    \begin{equation*}
        H : = \{x\gtrsim t^{\frac{1}{3}} \},
    \end{equation*}
    where the solution is oscillatory with dispersive decay.
    \item The self-similar region
    \begin{equation*}
        S : = \{|x|\lesssim t^{\frac{1}{3}} \},
    \end{equation*}
    where the solution essentially looks like a bump function with $t^{\frac{1}{3}}$ decay.
    \item The elliptic region
    \begin{equation*}
        E : = \{x\lesssim -t^{\frac{1}{3}} \},
    \end{equation*}
    where the solution no longer has oscillatory asymptotic behavior, and
consequently has better decay.
\end{enumerate}
By the stationary phase type argument, one can show
\begin{lemma}[Stationary phase, \cite{MR1232192}]
\begin{equation*}
 \| e^{t\partial_{x}^3} \|_{L^1 \rightarrow L^{\infty}}  \lesssim t^{-\frac{1}{3}} .
\end{equation*}
\end{lemma}
Consider the operator
\begin{equation*}
    L = x + 3t\partial_x^2,
\end{equation*}
it satisfies the following properties:
\begin{equation*}
    [\partial_x, L]=1, \qquad [L, \partial_t -\partial_x^3] = 0,
\end{equation*}
which further indicates that
\begin{equation*}
     \| L\phi(t)\|_{L^2} = \| x\phi(0)\|_{L^2}
\end{equation*}
is a conserved quantity. 
For the nonlinear equation \eqref{HBO}, we define the the nonlinear counterpart $L^{NL}$ of $L$:
\begin{equation}
    L^{NL}\phi = x\phi + 3t\phi_{xx} - \frac{3t}{4}\phi^3 - \frac{3t}{4}[\phi H\phi_x +H(\phi\phi_x)]. \label{NonlinearL}
\end{equation}\\

\textit{The Strichartz estimates.}
For the inhomogeneous equation 
\begin{equation}
    \phi_t - \phi_{xxx} = f, \quad \phi(0) = \phi_0, \label{InhomoAiry}
\end{equation}
we define Strichartz space $\mathcal{S}$ associated to the $L^2$ flow by
\begin{equation*}
    \mathcal{S} = L_t^\infty L_x^2\cap L_t^4\dot{W}^{\frac{1}{4},\infty},
\end{equation*}
together with its dual space
\begin{equation*}
    \mathcal{S}^{'} = L^1_t L_x^2 + L_t^{\frac{4}{3}}\dot{W}^{-\frac{1}{4},1}.
\end{equation*}
We also define
\begin{equation*}
    \mathcal{S}^s = (1+D^2)^{-\frac{s}{2}}\mathcal{S}
\end{equation*}
to denote the similar spaces associated to the flow in $H^s$.
The Strichartz estimates of \eqref{InhomoAiry} in the $L^2$ setting are summarized in the following:
\begin{lemma}[Strichartz estimate,\cite{MR2096258}]
\begin{equation*}
    \| \phi \|_{\mathcal{S}}  \lesssim
    \| \phi_0\|_{L^2}+\| f\|_{\mathcal{S}^{'}}.
\end{equation*}
\end{lemma}
The Besov type of above estimates is
\begin{equation*}
     \| \phi \|_{l^2 \mathcal{S}}  \lesssim \| \phi_0\|_{L^2}+\| f\|_{l^2 \mathcal{S}^{'}},
\end{equation*}
where
\begin{equation*}
\|\phi \|^2_{l^2 \mathcal{S}} = \sum_k \|\phi_k \|^2_{\mathcal{S}}, \qquad \|\phi \|_{l^2 \mathcal{S}^{'}}^2 = \sum_k \|\phi_k \|^2_{\mathcal{S}^{'}}.
\end{equation*}
There are also bilinear Strichartz estimate.
\begin{lemma}[Bilinear Strichartz estimate]
\begin{equation*}
   \|e^{t\partial_x^3 }P_j f \cdot e^{t\partial_x^3}P_k g \|_{L_{t,x}^2} \lesssim 2^{-max \{j,k\}} \|f \|_{L_{x}^2} \|g \|_{L_{x}^2}
\end{equation*}
for $|j-k|>2$, and for the case $|j-k|\leq 1$ as long as $P_j f$ and $P_k g$ have $O(2^k)$ frequency separation between their supports.
\end{lemma}
\begin{proof}
\begin{equation*}
    e^{t\partial_x^3 }P_j f \cdot e^{t\partial_x^3}P_k g = \int \hat{f}_j(\xi_1)\hat{g_k}(\xi_2)e^{i[(\xi_1+\xi_2)x-(\xi_1^3+\xi_2^3)t]}d\xi_1 d\xi_2 .
\end{equation*}
Then by duality, we can choose $\|F(t,x) \|_{L^2_{t,x}} = 1$, and a change of variable $\omega = \xi_1^3+\xi^3_2, \eta = \xi_1+\xi_2$, then the Jacobian $\frac{\partial(\omega , \eta)}{\partial(\xi_1 , \xi_2)} = 3(\xi_1^2 - \xi_2^2)$.
\begin{align*}
     \|e^{t\partial_x^3 }P_j f \cdot e^{t\partial_x^3}P_k g \|_{L_{t,x}^2} = \sup \int  F(\xi_1^3+\xi_2^3, \xi_1+\xi_2)\hat{f}_j(\xi_1)\hat{g}_k(\xi_2)d\xi_1 d\xi_2 \\
     \leq \|F \|_{L^2_{t,x}}\left(\int\hat{f}^2_j(\xi_1)\hat{g}^2_k(\xi_2)\left|\frac{\partial(\omega , \eta)}{\partial(\xi_1 , \xi_2)}\right|^{-2}d\omega d\eta\right)^{\frac{1}{2}} \approx 2^{-max \{j,k\}} \|f \|_{L_{x}^2} \|g \|_{L_{x}^2}.
\end{align*}
\end{proof}

A natural corollary for inhomogeneous equations then follows.
\begin{corollary}
Let $f$ and $g$ be the solution of inhomogeneous airy equation with initial data $P_j f_0$, $P_k g_0$, and source terms {F,G}, such that either $|j-k|>2$ or $|j-k|\leq 1$ but $P_j f_0$ and $P_k g_0$ have $O(2^k)$ frequency separation between their supports. Then we have
\begin{equation*}
    \|f g\|_{L^2_{t,x}}\lesssim 2^{-max \{j,k\}} (\| P_j f_0\|_{L^2}+ \|F \|_{L_t^{\infty}L_x^2})(\| P_k g_0\|_{L^2}+ \|G \|_{L_t^{\infty}L_x^2}).
\end{equation*} \label{t:InhomoBound}
\end{corollary}

\section{Normal form analysis and estimates} \label{s:NormalForm}
In this section we establish apriori $L^2$ bounds for solutions for the Cauchy problem \eqref{HBO}.
The Sobolev index of $L^2$ space is greater than the critical Sobolev index $-\frac{1}{2}$. By the scale invariance, we can rescale the solution and work with solution for which the $L^2$ norm is small. In this case the lifespan is also rescaled, it is natural to consider the solution on the time interval $[-1,1]$.
Resonant analysis suggests that we can use the normal form method to replace a nonresonant quadratic nonlinearity by a higher order cubic nonlinearity.

\subsection{The modified quadratic normal form analysis}
We rewrite \eqref{HBO} as
\begin{equation}
    \phi_t - \phi_{xxx}= -\frac{3}{4}\phi^2\phi_x-\frac{3}{4}\phi_x H\phi_x -\frac{3}{4}\phi H\phi_{xx}-\frac{3}{4}H(\phi_{xx}\phi+ \phi_x \phi_x). \label{HBO1}
\end{equation}
Then we hope to eliminate the quadratic terms, and to transform \eqref{HBO1} further to $(\partial_t-\partial_x^3)\tilde{\phi} = Q^3(\phi,\phi,\phi)+ Q^4(\phi,\phi,\phi,\phi)$ type, where $Q^3$ and $Q^4$ are translation invariant multi-linear forms. 
Let \begin{equation*}
    \tilde{\phi} = \phi + B(\phi,\phi)
\end{equation*}
to be the normal form transformation that eliminate the quadratic terms, where $B(\phi,\phi)$ is a bilinear  translation invariant symmetric function of $\phi$.
A direct computation shows the explicit expression of $B(\phi, \phi)$,
\begin{equation*}
    B(\phi, \phi) = -\frac{1}{4}[H(\phi\cdot \partial_x^{-1}\phi )+ H\phi\cdot \partial_x^{-1}\phi],
\end{equation*}
which is not invertible at low frequencies. 
Therefore, we need more refined estimates at different frequencies.
If we apply the operator $P_k^+$($P_k^-$ can be treated a in similar way), where $k\geq 1$, to \eqref{HBO1}, we have,
\begin{equation}
    P_k^+\phi_t - P_k^+\phi_{xxx}= -\frac{3}{4}P_k^+(\phi^2\phi_x)-\frac{3}{4}P_k^+(\phi_x H\phi_x) -\frac{3}{4}P_k^+(\phi H\phi_{xx})+\frac{3}{4}iP_k^+(\phi_{xx}\phi)+ \frac{3}{4}iP_k^+(\phi_x \phi_x) . \label{PKHBO1}
\end{equation}
For each of the four quadratic terms on the right hand side of \eqref{PKHBO1}, we can add an additional term, and form the commutator structure.
\begin{equation*}
\begin{aligned}
    -\frac{3}{4}P_k^+(\phi_x H\phi_x) - \frac{3}{4}i \partial_x \phi_{<k}\partial_x \phi_k^+ &= -\frac{3}{4}P_k^+(\partial_x\phi_{\geq k}H\phi_x) - \frac{3}{4}[P_k^+, \partial_x \phi_{<k}]H\phi_x ,\\
    -\frac{3}{4}P_k^+(\phi H\phi_{xx}) - \frac{3}{4}i  \phi_{<k}\partial_{xx} \phi_k^+ &= -\frac{3}{4}P_k^+(\phi_{\geq k}H\phi_{xx}) - \frac{3}{4}[P_k^+,  \phi_{<k}]H\phi_{xx} ,\\
    \frac{3}{4}i P_k^+(\phi_x^2)-\frac{3}{4}i\partial_x \phi_{<k}\partial_x \phi_k^+ &= \frac{3}{4}iP_k^+(\partial_x\phi_{\geq k}\phi_x) + \frac{3}{4}i[P_k^+ , \partial_x \phi_{<k}]\phi_x ,\\
    \frac{3}{4}i P_k^{+}(\phi_{xx}\phi)-\frac{3}{4}i\phi_{<k}\partial_{xx} \phi_k^+ &= \frac{3}{4}iP_k^{+}(\phi_{\geq k}\phi_{xx}) + \frac{3}{4}i[P_k^+ , \phi_{<k}]\phi_{xx} .
\end{aligned}
\end{equation*}

In this way, define an operator
\begin{equation}
    A_{TBO}^{k,+} = \partial_t - \partial_x^3 - \frac{3}{2}i \partial_x \phi_{<k}\partial_x - \frac{3}{2}i  \phi_{<k}\partial_{xx}, \label{DefAHOB}
\end{equation}
we can the rewrite the equation \eqref{PKHBO1} as
\begin{align*}
   A_{TBO}^{k,+}\phi_k^+ =  -\frac{3}{4}P_k^+(\phi^2\phi_x)  -\frac{3}{4}P_k^+(\partial_x\phi_{\geq k}H\phi_x)  -\frac{3}{4}P_k^+(\phi_{\geq k}H\phi_{xx})  +\frac{3}{4}iP_k^{+}(\partial_x\phi_{\geq k}\phi_x) +\frac{3}{4}iP_k^{+}(\phi_{\geq k}\phi_{xx}) \nonumber\\
    - \frac{3}{4}[P_k^+, \partial_x \phi_{<k}]H\phi_x - \frac{3}{4}[P_k^+,  \phi_{<k}]H\phi_{xx}  + \frac{3}{4}i[P_k^+ , \partial_x \phi_{<k}]\phi_x  + \frac{3}{4}i[P_k^+ , \phi_{<k}]\phi_{xx} .
\end{align*}
We further define
\begin{equation}
    \tilde{\phi}_k^+ = \phi_k^+ + B_k(\phi, \phi), \label{NormalFormK}
\end{equation} 
and use this normal form transformation to eliminate the quadratic term on the right.
Such a transformation is similarly computed and is given by the expression
\begin{equation}
    B_k(\phi, \phi) = \frac{1}{4}[iP_k^+(\phi\cdot \partial_x^{-1}\phi)-P_k^+(H\phi\cdot\partial_x^{-1}\phi) +2i\partial_x^{-1}(P_{<k}\phi)P_k^+\phi]. \label{BilinearK}
\end{equation}
To show that the bilinear form \eqref{BilinearK} is well-defined and the low frequency components are cancelled, we rewrite \eqref{BilinearK} using commutator structure.  
\begin{equation*}
    B_k(\phi, \phi) = \frac{1}{4}\{-[HP_k^+, \partial_x^{-1}\phi_{<k}]\phi - [P_k^+, \partial_x^{-1}\phi_{<k}]H\phi + iP_k^+(\phi \partial_x^{-1}\phi_{\geq k})-P_k^+(H\phi \cdot \partial_x^{-1}\phi_{\geq k})\}.
\end{equation*}
By \eqref{commutator},  we are able to shift a derivative from the high-frequency function $\phi$ or $H\phi$ to the low-frequency function $\partial_x^{-1}\phi_{<k}$.  
This shift ensures that all such commutator terms are easily estimated.
Then we have
\begin{equation}
     A_{TBO}^{k,+}\tilde{\phi}_k^+ = Q_k^3(\phi, \phi, \phi) +Q_k^4(\phi, \phi, \phi, \phi), \label{AHBOkEqn}
\end{equation}
where $Q_k^3(\phi, \phi, \phi)$ contains only cubic terms in $\phi$, and $Q_k^4(\phi, \phi, \phi, \phi)$ contains only quartic terms.
For low frequency case $k = 0$, we need no paradifferential component, and also we want to avoid the operator $P_0^+$ which does not have a smooth symbol. 
We simply write
\begin{align}
    (\partial_t - \partial_x^3)\phi_0 = & -\frac{3}{4}P_0(\phi^2\phi_x)-\frac{3}{4}P_0(\partial_x \phi_0H\phi_x)-\frac{3}{4}P_0(\phi_0H\phi_{xx}) - \frac{3}{4}P_0H(\phi_0\phi_{xx})  -\frac{3}{4}P_0H(\partial_x \phi_0 \phi_x) \nonumber\\
   &- \frac{3}{4}P_0(\phi_{>0}H\phi_{xx})-\frac{3}{4}P_0(\partial_x \phi_{>0}H\phi_x)-\frac{3}{4}P_0H(\phi_{>0}\phi_{xx})-\frac{3}{4}P_0H(\partial_x \phi_{>0}\phi_x).  \label{HBO0}
\end{align}
We only need to eliminate the last four terms on the right hand side of \eqref{HBO0} using normal form transformation. 
The other quadratic terms on the right is purely at low frequency and will play a perturbative role. 
An easy computation shows that
\begin{equation*}
    B_0(\phi, \phi) = -\frac{1}{4}[P_0(\partial_x^{-1}\phi_{>0}\cdot H\phi) + P_0H(\partial_x^{-1}\phi_{>0}\cdot H\phi)].
\end{equation*}
Next, we perform gauge transformation to simplify the terms on the left hand side of the equality. 
Note that many terms on the left have low frequency parts, so we choose the gauge transform 
\begin{equation}
    \psi_k^+ := \tilde{\phi}_k^+\cdot e^{-i\Phi_{<k}}. \label{GaugeTransform}
\end{equation}
Here $\Phi(t,x)$ is a real auxiliary function to be determined later.
The gauge transform \eqref{GaugeTransform} is similar to the Cole-Hopf transformation.
One observes that the expression  works as a normal form transformation, in the sense that it removes the paradifferential quadratic terms. 
The difference is that the exponential of an imaginary function is a bounded transformation, whereas the corresponding quadratic normal form is not. 
One also sees the difference reflected at the level of cubic or higher order terms
obtained after implementing these transformations.

 By applying this gauge transform \eqref{GaugeTransform}, we rewrite the equation \eqref{AHBOkEqn} as a
nonlinear Airy type equation for our final normal form variable $\psi_k^+$, with essentially cubic, quartic and quintic nonlinear terms:
\begin{equation}
    (\partial_t -\partial_x^3)\psi_k^+ = [\tilde{Q}_k^3(\phi,\phi,\phi)+\tilde{Q}_k^4(\phi,\phi,\phi,\phi)+\tilde{Q}_k^5(\phi,\phi,\phi,\phi,\phi)]e^{i\Phi_{<k}}, \label{NonlinearAiry}
\end{equation}
where $\tilde{Q}_k^3$, $\tilde{Q}_k^4$ and $\tilde{Q}_k^5$ contain only cubic, quartic, respectively quintic terms.
Obviously the exponent factor $e^{i\Phi_{<k}}$ contributes nothing in $L^2$ norm.

Direct computation shows
\begin{equation*}
\begin{aligned}
\partial_t\tilde{\phi}_k^+ &=\partial_t(\psi_k^+ e^{i\Phi_{<k}})
    = (\partial_t \psi_k^+ + i \partial_t
    \Phi_{<k}\psi_k^+)e^{i\Phi_{<k}} ,\\
    &\\
    \partial_x\tilde{\phi}_k^+  &=\partial_x(\psi_k^+ e^{i\Phi_{<k}})
    = (\partial_x \psi_k^+ + i \partial_x
    \Phi_{<k}\psi_k^+)e^{i\Phi_{<k}},\\
    & \\
    \partial_{xx}\tilde{\phi}_k^+ &= [\partial_{xx}\psi_k^+ +2i\partial_x \Phi_{<k}\partial_x\psi_k^+ +i \partial_{xx}\Phi_{<k}\psi_k^+ - (\partial_x \Phi_{<k})^2\psi_k^+]e^{i\Phi_{<k}},\\
    & \\
    \partial_{x}^3\tilde{\phi}_k^+ &= [\partial_x^3 \psi_k^+ + 2i\partial_{xx}\Phi_{<k}\partial_x\psi_k^+ +2i\partial_{x}\Phi_{<k}\partial_{xx}\psi_k^+ +i \partial_x^3\Phi_{<k}\psi_k^+ \\
    &+i\partial_x^2\Phi_{<k}\partial_x \psi_k^+ -2\partial_x\Phi_{<k}\partial_{xx}\Phi_{<k}\psi_k^+
     - (\partial_x \Phi_{<k})^2\partial_x \psi_k^+ + i\partial_x \Phi_{<k}\partial_x^2\psi_k^+ \\
     &-2(\partial_x \Phi_{<k})^2\partial_x\psi_k^+  -\partial_{xx} \Phi_{<k}\partial_x\Phi_{<k}\psi_k^+ - i(\partial_x \Phi_{<k})^3 \psi_k^+]e^{i\Phi_{<k}}.
\end{aligned}
\end{equation*}
By plugging in above terms into the definition of $A_{TBO}^{k,+}$ \eqref{DefAHOB}, we have by using \eqref{phirelation},
\begin{align*}
   e^{-i\Phi_{<k}} A_{TBO}^{k,+}\tilde{\phi}_k^+ &= (\partial_t - \partial_x^3)\psi_k^+ +( - \frac{3}{2}i\phi_{<k} - 3i\partial_x \Phi_{<k})\partial_x^2\psi_k^+ \\
   &+[-3i\partial_x^2\Phi_{<k}+ 3(\partial_x \Phi_{<k})^2+3\phi_{<k}\partial_{x}\Phi_{<k}-\frac{3}{2}i\partial_x\phi_{<k}]\partial_x \psi_k^+ \\
   &+[i\partial_t \Phi_{<k}-i\partial_x^3\Phi_{<k}-2i(\partial_x \Phi_{<k})^3+\frac{3}{2}\partial_x \phi_{<k}\partial_x \Phi_{<k}]\psi_k^+. 
\end{align*} 
To eliminate $\partial_x^2 \psi_k^+$ the term, we let its coefficient to be 0, so that
\begin{equation}
    \phi = -2\partial_x\Phi, \label{phirelation}
\end{equation}
and by \eqref{HBO1}, $\Phi$ satisfies the equation
\begin{equation*}
    \partial_t \Phi = \partial_x^3\Phi -(\partial_x \Phi)^3+\frac{3}{2}[\partial_x\Phi H\partial_x^2\Phi + H(\partial_x \Phi \partial_x^2 \Phi)].
    \end{equation*}
The coefficient for $\partial_x \psi_k^+$ terms becomes 
\[-\frac{3}{4}\phi_{<k}^2, 
\]
and the coefficient for $\psi_k^+$ terms becomes
\[
-\frac{3}{4}(\partial_x\phi_{<k})^2 +\frac{3}{8}iP_{<k}[\phi^3+\phi H\phi_x +H(\phi\phi_x)].
\]

Next, we compute each source term $Q^3_k$ and $\tilde{Q}^i_k$, $i = 3,4,5$.
They are given by
\begin{align*}
   Q_k^3(\phi, \phi, \phi) &= -\frac{1}{4}P_k^+\partial_x(\phi^3)+\frac{3}{4}B_k\big(\partial_x(\phi H\phi_x), \phi\big)+\frac{3}{4}B_k\big(\phi, \partial_x(\phi H\phi_x)\big) + \frac{3}{4}B_k\big(\partial_x H(\phi_x \phi), \phi\big)\\
    & \qquad + \frac{3}{4}B_k\big(\phi,\partial_x H(\phi_x \phi)\big)-\frac{3}{2}i\partial_x\phi_{<k}\partial_{x}B_k(\phi, \phi) - \frac{3}{2}i  \phi_{<k} \partial_x^2B_k(\phi, \phi),  \\
    &\\
  Q_k^4(\phi, \phi, \phi, \phi) &= \frac{1}{4}B_k\big(\partial_x(\phi^3),\phi\big)+\frac{1}{4}B_k\big(\phi, \partial_x(\phi^3)\big), \\
   &\\
  \tilde{Q}_k^3(\phi, \phi, \phi) &=  Q_k^3(\phi, \phi, \phi)+ \frac{3}{4}\phi_{<k}^2\partial_x\phi_k^+ +\frac{3}{4}(\partial_x\phi_{<k})^2 \phi_k^+ -\frac{3}{8}iP_{<k}\big(\phi H\phi_x+H(\phi\phi_x)\big)\phi_k^+ ,\\
  &\\
   \tilde{Q}_k^4(\phi, \phi, \phi, \phi) &= Q_k^4(\phi, \phi, \phi, \phi)+\frac{3}{8}(i+2)\phi_{<k}^3\phi_k^+  -\frac{3}{8}iP_{<k}(\phi^3)\phi_k^+ + \frac{3}{4}\phi_{<k}^2\partial_xB_k(\phi, \phi)\\
   &\qquad +\frac{3}{4}(\partial_x\phi_{<k})^2 B_k(\phi, \phi) -\frac{3}{8}iP_{<k}\big(\phi H\phi_x + H(\phi \phi_x)\big)B_k(\phi, \phi),\\
    &\\
   \tilde{Q}_k^5(\phi, \phi, \phi, \phi, \phi) &= \frac{3}{8}(i+2)\phi_{<k}^3B_k(\phi, \phi) -\frac{3}{8}iP_{<k}(\phi^3)B_k(\phi, \phi).
\end{align*}
The case $k = 0$ is special here as well, we simple set $\psi_0 = \tilde{\phi}_0$, and use the original equation.
Owing to the commutator structure, we can easily obtain a similar estimate.

This concludes the algebraic part of the analysis. Our next goal is study the analytic
properties of those multilinear forms:
\begin{lemma}
The quadratic form $B_k$ can be expressed as
\begin{equation*}
    B_k(\phi, \phi) = 2^{-k}L_k(H\phi_{<k}, \phi_k)+ 2^{-k}L_k(\phi_{<k}, H\phi_k) + \sum_{j \geq k}2^{-j}L_k(\phi_j, \phi_j) = 2^{-k}L_k(\phi, \phi).
\end{equation*}
For multilinear expressions $Q_k^i$ and $\tilde{Q}_k^i$, if we define $\Bar{\phi}$ to be either $\phi$ or $H\phi$, then they can be written in translation invariant forms of the type
\begin{align*}
    &Q_k^3(\phi, \phi, \phi) = 2^kL_k(\Bar{\phi}, \Bar{\phi}, \Bar{\phi}),\\
    &Q_k^4(\phi, \phi, \phi, \phi) = L_k(\Bar{\phi},\Bar{\phi}, \Bar{\phi}, \Bar{\phi}),\\
    &\tilde{Q}_k^3(\phi, \phi, \phi) =2^k L_k(\Bar{\phi}, \Bar{\phi}, \Bar{\phi}),\\
&\tilde{Q}_k^4(\phi, \phi, \phi, \phi) = L_k(\Bar{\phi},\Bar{\phi}, \Bar{\phi}, \Bar{\phi}),\\
&\tilde{Q}_k^5(\phi,\phi, \phi, \phi, \phi) = 2^{-k}L_k(\Bar{\phi},\Bar{\phi}, \Bar{\phi}, \Bar{\phi},\Bar{\phi}),
      \end{align*}
all with output frequency $2^k$.  \label{t:BilinearK}
\end{lemma}
\begin{proof}
For first 2 terms of $B_k(\phi, \phi)$ we use the properties of commutator structure to shift a derivative from high-frequency to low-frequency function.
For the two remaining terms we split the unlocalized $\phi$ factor into
$\phi_{<k}+\phi_{\geq k}$. The terms involving $\phi_{<k}$ become a part of $2^{-k}L_k(H\phi_{<k}, \phi_k)+ 2^{-k}L_k(\phi_{<k}, H\phi_k)$ since we have
\begin{equation*}
    \partial_x^{-1}\phi_{\geq k}= 2^{-k}L(\phi).
\end{equation*}
For the remaining bilinear terms in $\phi_{\geq k}$ the frequencies of two inputs must be balanced at some frequency $2^j$ with $j \geq k$.
The rest expressions follow from direct computations and that of $B_k(\phi, \phi)$. 
\par Next, we consider each term of $Q_k^3(\phi, \phi, \phi)$ separately. 
The first term is direct, since the spacial derivative $\partial_x$ is equivalent to the factor $2^k$ under the frequency projection operator $P_k$.
For next four terms, we can use the expression for $B_k(\phi, \phi)$, for some terms we may use the properties of commutator structure.
For example, 
\begin{equation*}
 \begin{aligned}
    B_k(\partial_x(\phi H\phi_x), \phi) = -\frac{1}{4}\{ [HP_k^+, P_{<k}(\phi H\phi_x)]\phi + [P_k^+, P_{<k}(\phi H\phi_x)]\phi   \\
    -iP_k[\partial_x(\phi H\phi_x)\cdot\partial_x^{-1}\phi_{\geq k}] +P_k^+[H\partial_x(\phi H\phi_x)\cdot\partial_x^{-1}\phi_{\geq k}] \}.
 \end{aligned}
\end{equation*}
The first two equals $L_k(\Bar{\phi}, \Bar{\phi}, \Bar{\phi}_x)$.
For the third term
\begin{equation*}
    P_k[\partial_x(\phi H\phi_x)\cdot\partial_x^{-1}\phi_{\geq k}] =P_k\partial_x[\phi H\phi_x\cdot \partial_x^{-1}\phi_{\geq k}] - P_k[\phi H\phi_x \phi_{\geq k}],
\end{equation*}
it equals $L_k(\Bar{\phi}, \Bar{\phi}, \Bar{\phi}_x)$ due to the fact $\partial_x^{-1}\phi_{\geq k} = 2^{-k}\phi$.
The fourth term is similar to the third term.
As for the last four terms of $Q_k^3(\phi, \phi, \phi)$, it follows directly from the above expression of $B_k(\phi,\phi)$.
Note that 
\begin{equation*}
    L_k(\Bar{\phi},\Bar{\phi},\Bar{\phi}_x) = L_k(\partial_x(\Bar{\phi})^3) = 2^k L_k(\Bar{\phi},\Bar{\phi},\Bar{\phi}),
\end{equation*}
we can obtain the final expression for $Q_k^3(\phi, \phi, \phi)$. 
The proof for other $Q_k^i$ or $\tilde{Q}^i_k$ is similar.
\end{proof}

\subsection{Bounds using the bootstrap argument}
We already know from \cite{MR3005535} that the equation \eqref{HBO} is $H^{2}$ globally well-posed.
We prove further an apriori estimate using a standard continuity argument.
We first assume that the initial data is small in $L^2$ norm,
\begin{equation}
    \| \phi(0)\|_{L^2} \leq \epsilon, \label{InitialCondition}
\end{equation}
for a small positive constant $\epsilon$.
Then our main apriori estimate is as follows:
\begin{theorem}
Let $\{ c_k\}_{k=0}^\infty \in l^2$, and $\epsilon c_k$ be a frequency envelope of $\phi$ in $L^2$, so that the  initial condition \eqref{InitialCondition} is strengthen to
\begin{equation}
    \|\phi_k(0) \|_{L^2} \lesssim  \epsilon c_k, \label{InitialConditionK}
\end{equation}
as explained in Section \ref{s:Definition}.
Let $0<t\leq 1$ and $j$, $k\in \mathbb{N}$ , define
\begin{equation}
    M(t) :=  \sup_j c_j^{-2} \| P_j\phi\|^2_{\mathcal{S}^0[0,t]}+  \sup_{j\neq k} \sup_{y\in \mathbb{R}} c_j^{-1}c_k^{-1}2^{\text{max}\{ j,k\}}\|\phi_jT_y \phi_k \|_{L^2[0,t]}.
\end{equation}
For some fixed $t_0$, $ M(t_0) \lesssim \epsilon^2.$ \label{Theorem4}
\end{theorem}
To prove this, for a large constant $C$ which is independent of $\epsilon$,
we use a bootstrap argument by assuming that 
\begin{equation*}
    M(t_0) \leq C^2 \epsilon^2,
\end{equation*}
and then to show the following improved result
\begin{equation*}
    M(t_0) \lesssim  \epsilon^2 + C^6\epsilon^6.
\end{equation*}
Given the bootstrap assumption, we have the starting estimate
\begin{equation}
    \| \phi_k\|_{\mathcal{S}} \lesssim Cc_k\epsilon, \label{BootstrapOno}
\end{equation}
and
\begin{equation}
    \|\phi_j T_y \phi_k  \|_{L^2} \lesssim 2^{- max\{j,k \}}C^2 \epsilon^2 c_j c_k, \qquad j\neq k,\quad y\in \mathbb{R},   \label{Bootstraptwo}
\end{equation}
where in the bilinear case, we also allow $j = k$ provided the two localization multipliers are at least $2^{k-4}$ separated.
In the following step, we will consider first the bound for $\psi_k^+$, 
and obtain the bound for $\phi_k^+$ according to the normal form analysis.
\begin{lemma}\label{t:psik}
Assume \eqref{InitialConditionK}, then we have 
\begin{equation*}
    \|\psi_k^+(0) \|_{L^2} \lesssim  \epsilon c_k.
\end{equation*} 
\end{lemma}
\begin{proof}
Obviously, the $L^2$ norms of $\psi_k^+$ and $\tilde{\phi}_k^+$ are the same, it only suffices to prove the estimate for $\tilde{\phi}_k^+$.
Thus, we need to prove that $L^2$ norm of $\tilde{\phi}_k^+$ is comparable with the $L^2$ norm of
$\phi_k^+$.
The two variables are related via the relation \eqref{NormalFormK}, and we reduce our problem to the study of the $L^2$ bound for the bilinear form $B_k(\phi, \phi)$.
From lemma \ref{t:BilinearK}, we know that
\begin{equation*}
    B_k\left(\phi(0), \phi(0)\right) = 2^{-k}L_k\left(\Bar{\phi}_{<k}(0), \Bar{\phi}_k(0)\right)+ \sum_{j \geq k}2^{-j}L_k\left(\phi_j(0), \phi_j(0)\right).
\end{equation*}
If we apply the Bernstein's inequality,
\begin{equation*}
    \| \phi_{<k}\|_{L^\infty_x}\leq 2^{\frac{k}{2}}\| \phi_{<k}\|_{L^2_x},
\end{equation*}
and the H{\"o}lder's inequality, then we have
\begin{equation*}
    \|2^{-k}L_k\left(\Bar{\phi}_{<k}(0),\Bar{\phi}_k(0)\right)  \|_{L^2}\lesssim 2^{-\frac{k}{2}}\|\phi_{<k}(0) \|_{L^2}\|\phi_k(0) \|_{L^2}\leq 2^{-\frac{k}{2}}c_k\epsilon^2,
\end{equation*}
and 
\begin{equation*}
    \|  \sum_{j \geq k}2^{-j}L_k\left(\phi_j(0), \phi_j(0)\right) \|_{L^2} \lesssim \sum_{j\geq k}2^{-\frac{j}{2}}\cdot \epsilon \|\phi_j(0) \|_{L^2}\lesssim \sum_{j\geq k} 2^{-\frac{j}{2}}\epsilon^2\cdot c_j\lesssim 2^{-\frac{k}{2}}\cdot c_k\epsilon^2.
    \end{equation*}
    This concludes the proof.
\end{proof}
In a similar fashion, we can obtain the bound 
\begin{equation}
    \|B_k(\phi, \phi)\|_{\mathcal{S}} \lesssim 2^{-\frac{k}{2}}\| \phi_{<k}\|_{L_t^\infty L_x^2}\|\phi_k\|_{\mathcal{S}} \lesssim 2^{-\frac{k}{2}}C^2c_k\epsilon^2. \label{BilinearKS}
\end{equation}
\par For next part, we consider the right hand side of the equation \eqref{NonlinearAiry}.

\begin{lemma} \label{t:tildeQ345}
Assume the bootstrap assumptions \eqref{BootstrapOno} and \eqref{Bootstraptwo}, then we have
\begin{equation*}
    \|\tilde{Q}_k^3 \|_{L_t^1 L_x^2}+  \|\tilde{Q}_k^4 \|_{L_t^1 L_x^2}+  \|\tilde{Q}_k^5 \|_{L_t^1 L_x^2} \lesssim C^3\epsilon^3(1+C\epsilon + C^2\epsilon^2).
\end{equation*} 
\end{lemma}
\begin{proof}
First for $\tilde{Q}_k^3$ terms, let us assume that the first entry of $L_k$ is localized at frequency $2^{k_1}$, the second at frequency $2^{k_2}$, and finally the third one is at frequency $2^{k_3}$, with $k_1\leq k_2 \leq k_3$. 
As the output is at frequency $2^k$, there are three possible cases:
\begin{itemize}
    \item If $2^{k_1} = 2^{k_2} = 2^{k_3}>2^k$, then we have $\phi_{k_1}$ and $\phi_{k_3}$ must have at least a $2^{k_3-4}$ frequency separation.
    We use the bilinear Strichartz estimate for the $2^{k_3-4}$ frequency separated solutions, and the Strichartz inequality for the remaining term
    \begin{align*}
    \|\tilde{Q}_k^3(\phi_{k_1}, \phi_{k_2},\phi_{k_3})\|_{L_t^{\frac 4 3}L_x^2} &\lesssim 2^{-\frac{k_2}{4}}\| L(\phi_{k_1}, \phi_{k_3})\|_{L_{t,x}^2}\| \phi_{k_2}\|_{L_t^4\dot{W}_x^{\frac{1}{4},\infty}} \\
    &\lesssim 2^{-k_3 -\frac{k_2}{4}} C^3\epsilon^3 c_{k_1} c_{k_2} c_{k_3} \lesssim 2^{-k}C^3c_k^3\epsilon^3.
    \end{align*}
    \item If $2^{k}<2^{k_1} < 2^{k_2} < 2^{k_3}$, the estimate is similar to the first case.
    \item If $2^{k_1}=2^{k_2}=2^{k_3} \approx 2^{k}$,  we use directly the Strichartz estimate
    \begin{align*}
    \|\tilde{Q}_k^3(\phi_{k_1}, \phi_{k_2},\phi_{k_3})\|_{L_t^2L_x^2} &\lesssim 2^{-\frac{k_2+k_3}{4}} \|\phi_{k_1} \|_{L_t^\infty L_x^2} \| \phi_{k_2}\|_{L_t^4\dot{W}_x^{\frac{1}{4},\infty}}\cdot \| \phi_{k_3}\|_{L_t^4\dot{W}_x^{\frac{1}{4},\infty}} \\
    &\lesssim C^3\epsilon^3c_{k_1}c_{k_2}c_{k_3}2^{-\frac{k_2+k_3}{4}} \lesssim 2^{-\frac{k}{4}}C^3c_k^3\epsilon^3.
    \end{align*}
\end{itemize}

\par Then we consider $\tilde{Q}_k^4$ terms.
For $Q_k^4$ terms
\begin{equation*}
\begin{aligned}
    \|Q_k^4 \|_{L_t^2L_x^2} \lesssim & 2^{-k}\|L_k\left(P_{<k}(\partial_x\Bar{\phi}^3),\phi_k\right) \|_{L_t^2L_x^2}+ 2^{-k}\|L_k\left(P_{<k}(\phi),P_k\partial_x\Bar{\phi}^3\right) \|_{L_t^2L_x^2} \\ 
    &+\sum_{j\geq k}2^{-j}\|L_k\left(P_j\partial_x(\phi^3),\phi_j\right) \|_{L_t^2L_x^2} \lesssim c_kC^4\epsilon^4.
    \end{aligned}
\end{equation*}
For next term of $\tilde{Q}_k^4$, they can be estimated directly if we apply the $L_t^4\dot{W}_x^{\frac{1}{4},\infty}$ norm for $\phi_{<k}$ and $L_t^\infty L_x^2$ norm for $\phi_k^+$. The remaining four terms can be treated in a similar way.
\par $\tilde{Q}_k^5$ terms can be estimated directly.
\begin{equation*}
    \|\tilde{Q}_k^5 \|_{L_t^{\frac{4}{3}} L_x^2}\lesssim \|B_k(\phi,\phi) \|_{L_t^\infty L_x^2}\|\phi_{<k} \|^3_{L_t^4\dot{W}_x^{\frac{1}{4},\infty}}\lesssim C^5c_k\epsilon^5.
\end{equation*}
\end{proof}
Then combining Lemma \ref{t:psik} and Lemma \ref{t:tildeQ345}, we have
\begin{equation*}
    \|\psi_k^+ \|_{\mathcal{S}} \lesssim \|\psi_k^+(0) \|_{L^2}+\|\tilde{Q}_k^3+\tilde{Q}_k^4+\tilde{Q}_k^5 \|_{L_t^1 L_x^2}\lesssim c_k\epsilon(1+C^2\epsilon^2).
\end{equation*}
Together with the bound for $\|B_k(\phi, \phi) \|_{\mathcal{S}}$ \eqref{BilinearKS}, we obtain the Strichartz component of the bootstrap argument.

\par We then consider the bilinear estimate in our bootstrap argument.
According to the bilinear Strichartz estimate for Airy-type equation,
if we define truncation operators $\tilde{P}_j$ and $\tilde{P}_k$ which
    still have $2^{max\{j,k\}}$ separated supports but whose symbols are identically $1$ in the support of $P_j$ , respectively $P_k$, then by Corollary \ref{t:InhomoBound},
\begin{equation*}
    \|\tilde{P}_j\psi_j^+ \cdot \tilde{P}_k \psi_k^+ \|_{L^2_{t,x}}\lesssim \epsilon^2 2^{-max\{j,k \}}c_j c_k(\epsilon^2 +C^6\epsilon^6).
\end{equation*}
At last, we need to transfer the bound to $\phi_j^+ \phi_k^+$.
\begin{equation*}
 \tilde{P}_j\psi_j^+ \cdot \tilde{P}_k\psi_k^+ - \phi_j^+ e^{-i\Phi_{<j}}\phi_k^+ e^{-i\Phi_{<k}}  = \tilde{P}_j\psi_j^+(\tilde{P}_k\psi_k^+ - \phi_k^+e^{-i\Phi_{<k}})+(\tilde{P}_j\psi_j^+ - \phi_j^+ e^{-i\Phi_{<j}})\phi_k^+ e^{-i\Phi_{<k}}.
\end{equation*}
For the first component, note that
\begin{equation}
    \|\psi_k^+ - \phi_k^+e^{-i\Phi_{<k}}\|_{L_t^4\dot{W}_x^{\frac{1}{4},\infty}} \lesssim 2^{-\frac{k}{2}}\|\phi_k \|_S\|\phi_{<k} \|_{L_t^\infty L_x^2}  +\sum_{j\geq k} 2^{-\frac{j}{2}}\|\phi_j \|_S\|\phi_{<j} \|_{L_t^\infty L_x^2}\lesssim 2^{-\frac{k}{2}}C^2\epsilon^2 c_k, \label{FirstComponent}
\end{equation}
\begin{equation*}
 \|\tilde{P}_j\psi_j^+(\tilde{P}_k\psi_k^+ - \phi_k^+e^{-i\Phi_{<k}}) \|_{L_t^\infty L_x^2}\lesssim 2^{-\frac{k}{4}}\| \psi_j^+\|_{L_t^\infty L_x^2} \|\psi_k^+ - \phi_k^+e^{-i\Phi_{<k}} \|_{L_t^4\dot{W}_x^{\frac{1}{4},\infty}}\lesssim 2^{-\frac{3k}{4}}C^3\epsilon^3c_j c_k.
\end{equation*}
For the second term, we rewrite it in a commutator structure, and use the property of commutator \eqref{commutator}.
\begin{equation*}
\begin{aligned}
\tilde{P}_j\psi_j^+ - \phi_j^+e^{-i\Phi_{<j}} & = (\tilde{P}_j-1)(\tilde{\phi}_j^+ e^{-i\Phi_{<j}}) +B_j(\phi, \phi)e^{-i\Phi_{<j}} \\
& = [\tilde{P}_j - 1, e^{-i\Phi_{<j}}]\phi_j^+ +(\tilde{P}_j -1)\left(B_j(\phi, \phi)e^{-i\Phi_{<j}}\right) -e^{-i\Phi_{<j}}\tilde{P}_jB_j(\phi, \phi) \\
 & = 2^{-j}L(\partial_{x}e^{-i\Phi_{<j}}, \phi_j^+)+L\left(B_j(\phi, \phi), e^{-i\Phi_{<j}}\right) \\
 & = 2^{-j}L(\phi_{<j}, \phi_j, e^{-i\Phi_{<j}})+ \sum_{l>j}2^{-l}L(\phi_l, \phi_l, e^{-i\Phi_{<j}}).
\end{aligned}
\end{equation*}
Thus we use the bilinear estimate and Bernstein's inequality to get
\begin{equation*}
    \begin{aligned}
\|(\tilde{P}_j\psi_j^+ - \phi_j^+e^{-i\Phi_{<j}})\phi_k^+ \|_{L_{t,x}^2} &\lesssim  2^{-j}\|L(\phi_{k}, \phi_j) \|_{L_t^2L_x^2}\|\phi_{<j}^+ \|_{L_{t,x}^\infty} + \sum_{l>j}2^{-l}\|L(\phi_l, \phi_k) \|_{L_t^2L_x^2}\|\phi_l \|_{L_{t,x}^\infty} \\
 \lesssim & 2^{-j-max\{j,k\}}C^2\epsilon^2c_jc_k2^{\frac{j}{2}}C\epsilon + \sum_{l>j}2^{-l-max\{l,k\}}C^3\epsilon^3 c_l^22^{\frac{l}{2}}c_k \lesssim 2^{-max\{j,k\}}c_j^2 c_kC^3\epsilon^3,
    \end{aligned}
\end{equation*}
which is a small error term.
In this way, we complete the bootstrap argument.

\section{Linearized third order Benjamin Ono equation} \label{s:LinearizedHBO}
The linearized equation of \eqref{HBO} is 
\begin{equation}
    v_t - v_{xxx} + \frac{3}{2}\phi\phi_x v + \frac{3}{4}\phi^2 v_x +\frac{3}{4}v_x H\phi_x+\frac{3}{4}\phi_x Hv_x + \frac{3}{4}v H\phi_{xx}+\frac{3}{4}\phi Hv_{xx}+\frac{3}{4}H(v_{xx}\phi + \phi_{xx} v + 2v_x\phi_x)=0. \label{LHBO}
\end{equation}
Understanding the properties of the linearized flow is critical for any local well-posedness
result.
In this section, we show the following result:
\begin{proposition}
Let $\phi$ be a $H^3$ solution to the third order Benjamin-Ono equation in $[0,1]$ with small $H^{\frac{1}{2}}$ norm. Then the linearized equation \eqref{LHBO} is well-posed in $H^{-\frac{1}{2}}$, and satisfies the inequality
\begin{equation}
    \|v \|_{C(0,1; H^{-\frac{1}{2}})} \lesssim \| v_0\|_{H^{-\frac{1}{2}}}. \label{e:LinearizedBound}
\end{equation}
with a universal implicit constant that does not depend on the $H^3$ norm of $\phi$. \label{t:LinearizedBound}
\end{proposition}
The rest of the section is devoted to the proof of this proposition. 
\par First, we show that the $L^2$ norm of this linearized equation can be controlled by the initial data given that $\phi$ is a $H^3$ solution of \eqref{HBO}. 
Multiplying by $v$, and using the integration by part tricks and the properties of Hilbert transform \eqref{HilbertOne} \eqref{HilbertTwo}, we see that
term like $vv_{xxx}$ vanishes.
\begin{equation*}
\begin{aligned}
    \frac{1}{2}\cdot\frac{d}{dt}\|v \|^2_{L^2} &= -\frac{3}{4}\int [\partial_x(\phi^2)v^2  +\frac{1}{2}\phi^2 \partial_x(v^2)-vv_xH\phi_x- \phi_xvHv_x- v^2H\phi_{xx}] dx \\ &\lesssim (\|\phi_{xx} \|_{L^\infty} + \|\partial_x(\phi^2) \|_{L^\infty})\|v \|_{L^2}^2.
\end{aligned}
\end{equation*}
$\| \phi_{xx}\|_{L^\infty}$ can be bounded by $\|\phi \|_{H^3}$ by the Sobolev embedding.
Here, the most difficult term is $\phi_xHv\cdot v_x$.
Note that according to the convolution identity of Hilbert transform \eqref{HilbertThree},
\begin{equation*}
Hv\cdot v_x = -vHv_x+H[v_x\cdot v]-H[Hv_x\cdot Hv].
\end{equation*}
As a consequence,
\begin{equation*}
\begin{aligned}
    \int \phi_x\cdot Hv_x\cdot v dx = -\int \phi_{xx}Hv\cdot v dx -\int \phi_x Hv\cdot v_x dx  \\
    = -\int \phi_{xx}Hv\cdot v dx -\frac{1}{2}\int \phi_x(H[v_x\cdot v]-H[Hv_x\cdot Hv])dx.
\end{aligned}
\end{equation*}
These last three terms is easy to handle. By Gronwall's inequality,
\begin{equation*}
    \|v \|_{L_t^\infty L_x^2}\lesssim \|v_0 \|_{L_x^2},
\end{equation*}
which shows the uniqueness of $v$ in $L^2$.
For the backward adjoint linearized equation
\begin{equation*}
\partial_t w - w_{xxx} -\frac{3}{2}\phi\phi_x w + \frac{3}{4}(\phi^2w)_x + \frac{3}{4}w_xH\phi_x +\frac{3}{4}H(w_x\phi)_x+\frac{3}{4}Hw_{xx}\phi = 0,
\end{equation*}
with initial data
\begin{equation*}
    w(1) = w_1,
\end{equation*}
we similarly have
\begin{equation*}
    \| w\|_{L^\infty_t L_x^2} \lesssim \|w_1 \|_{L_x^2},
\end{equation*}
which proves the existence for the direct problem.
Following the method in \cite{MR3948114} by consider a equation of $v_1 = v_x$, we have that the linearized equation \eqref{LHBO} is well-posed in $H^1$.
\par Next, in order to prove Proposition \ref{t:LinearizedBound}, it suffices to show that the $H^1$
solutions $v$ constructed above satisfy the bound \eqref{e:LinearizedBound}. 
Suppose for initial condition,  $\| v(0)\|_{H^{-\frac{1}{2}}} \leq 1$, and consider a frequency envelop $\{ d_k\}_{k = 0}^\infty$ for $v(0)$ in $H^{-\frac{1}{2}}$ with the property $c_k \lesssim d_k$, where $c_k$ represents an $L^2$ frequency envelope for $\phi(0)$ as in the previous section.
Then the initial condition is strengthen to
\begin{equation*}
    \|v_k\|_{L^2} \leq 2^{\frac{k}{2}}d_k.
\end{equation*}
We aim to prove that the dyadic pieces $v_k$ of $v$ satisfy the estimates
\begin{equation*}
    \| v_k\|_{\mathcal{S}} \lesssim 2^{\frac{k}{2}}d_k,
\end{equation*}
as well as the bilinear $L^2$ estimates
\begin{equation*}
    \|L(v_j, \phi_k) \|_{L^2_{t,x}}\lesssim \epsilon d_j c_k 2^{\frac{j}{2}}\cdot 2^{-max\{j,k \}}.
\end{equation*}
To start with, we will use the bootstrap argument,  assuming 
\begin{gather}
    \| v_k\|_{\mathcal{S}} \lesssim C2^{\frac{k}{2}}d_k ,\label{AssumptionOne}\\
    \|L(v_j, \phi_k) \|_{L^2_{t,x}}\lesssim \epsilon Cd_j c_k 2^{\frac{j}{2}}\cdot 2^{-max\{j,k \}}. \label{AssumptionTwo}
\end{gather}
and prove that
\begin{gather}
\| v_k\|_{\mathcal{S}} \lesssim (1+\epsilon C)2^{\frac{k}{2}}d_k , \label{BootstrapOne}\\
    \|L(v_j, \phi_k) \|_{L^2_{t,x}}\lesssim \epsilon (1+\epsilon C) d_j c_k 2^{\frac{j}{2}}\cdot 2^{-max\{j,k \}}. \label{BootstrapTwo}
\end{gather}
We consider each frequency component, and proceed in the same manner as for the nonlinear equation, rewriting the linearized equation in paradifferential form as
\begin{equation}
\begin{aligned}
    A_{HBO}^{k,+}v_k^+ &= -\frac{3}{2}P_k^+(\phi\phi_x v) - \frac{3}{4}P_k^+(\phi^2 v_x) -\frac{3}{4}P_k^+(v_x H\phi_x)-\frac{3}{4}P_k^+(\phi_x Hv_x) \nonumber \\
    &\quad -\frac{3}{4}P_k^+(v H\phi_{xx})
    -\frac{3}{4}P_k^+(\phi Hv_{xx}) 
    +\frac{3}{4}iP_k^+(v_{xx}\phi + \phi_{xx} v + 2v_x\phi_x) \nonumber \\
    &\quad - \frac{3}{2}i \partial_x \phi_{<k}v_x - \frac{3}{2}i  \phi_{<k}v_{xx}.
\end{aligned}
\end{equation}
We then apply the normal form transformation to eliminate quadratic terms on the right.
The first 5 quadratic terms are eliminated by
\begin{equation}
     \frac{1}{4}[iP_k^+(\phi\cdot \partial_x^{-1}v) +iP_k^+(v\cdot \partial_x^{-1}\phi)-P_k^+(H\phi\cdot\partial_x^{-1}v)-P_k^+(Hv\cdot\partial_x^{-1}\phi)].\label{BklinOne}
\end{equation}
Then the last 2 terms are eliminated by 
\begin{equation}
    \frac{1}{2}i\partial_x^{-1}(P_{<k}\phi)P_k^+v. \label{BklinTwo}
\end{equation}
Note that if we add
\begin{equation*}
    \frac{1}{2}i\partial_x^{-1}(P_{<k}v)P_k^+\phi
\end{equation*}
to \eqref{BklinOne} and \eqref{BklinTwo}, then it becomes exactly $2B_k(v,\phi)$, where $B_k$ stands for the symmetric bilinear form associated to the quadratic form $B_k$ defined
in \eqref{BilinearK}.
Since $\partial_x^{-1}v$ is ill-defined at low frequency, we would exclude the frequencies $< 1$. 
Hence, our quadratic normal form is chosen to be
\begin{equation*}
    B_k^{lin}(\phi, v)= 2B_k(v,\phi) -  \frac{1}{2}i\partial_x^{-1}v_{(0,k)}P_k^+\phi .
\end{equation*}
In this way, define
\begin{equation*}
    \tilde{v}_k^+ = v_k^+ + B_k^{lin}(\phi, v).
\end{equation*}
Then equation for $\tilde{v}_k^+$ has the form
\begin{equation*}
    A_{HBO}^{k,+}\tilde{v}_k^+ = Q_k^{2,lin}(v_0, \phi_k) +Q_k^{3,lin}(\phi,\phi,v)+Q_k^{4,lin}(\phi,\phi,\phi,v).
\end{equation*}
Calculations show that
\begin{equation*}
\begin{aligned}
    Q_k^{2,lin}(v_0, \phi_k) &= - \frac{3}{2}i (\partial_x v_0\partial_x\phi_k^+ +v_0\partial_{xx}\phi_k^+ ),\\
    &\\
     Q_k^{3,lin}(\phi, \phi, v) &= -\frac{3}{4}P_k^+(\phi^2 v)_x+\frac{3}{4}B_k^{lin}\left((\phi H\phi_x)_x,v\right) +\frac{3}{4}B_k^{lin}(H(\phi \phi_x)_x,v)  \\ 
     &\quad +\frac{3}{4}B_k^{lin}(\phi, (vH\phi_x)_x)+\frac{3}{4}B_k^{lin}\left(\phi, (\phi Hv_x)_x\right)+\frac{3}{4}B_k^{lin}\left(\phi,H(v\phi)_{xx}\right)   \\
     &\quad +\frac{3}{2}i\partial_x \phi_{<k}\partial_xB_k^{lin}(\phi,v) - \frac{3}{2}i\phi_{<k}\partial_x^2B_k^{lin}(\phi, v), \\
     & \\
     Q_k^{4,lin}(\phi,\phi,\phi,v) &= \frac{1}{4}B_k^{lin}\left((\phi^3)_x,v\right)+\frac{3}{4}B_k^{lin}\left(\phi, (\phi^2 v)_x\right).
\end{aligned}
\end{equation*}
Again there is a straightforward adjustment in this analysis for the case $k =0$, following the model in Section~\ref{s:NormalForm}. 
This adds a trivial low frequency quadratic term on the right.
\par Finally, for $k>0$, we perform renomalization of $\tilde{v}_k^+$ by
\begin{equation*}
    w_k^+ := e^{-i\Phi_{<k}}\tilde{v}_k^+,
\end{equation*}
which in turn solves the inhomogeneous Airy-type equation
\begin{equation}
    (\partial_t -\partial_x^3)w_k^+ = [Q_k^{2,lin}(v_0, \phi_k^+)+ \tilde{Q}_k^{3,lin}(\phi, \phi, v)+\tilde{Q}_k^{4,lin}(\phi,\phi,\phi,v)+\tilde{Q}_k^{5,lin}(\phi,\phi,\phi,\phi,v)]e^{-i\Phi_{<k}}, \label{wkEqn}
\end{equation}
where
\begin{equation*}
\begin{aligned}
     \tilde{Q}_k^{3,lin}(\phi, \phi, v) &=  Q_k^{3,lin}(\phi, \phi, v)+ \frac{3}{4}(2i-1)\phi_{<k}^2\partial_x v_k^+ +\frac{3}{4}(\partial_x\phi_{<k})^2 v_k^+ -\frac{3}{8}iP_{<k}\big(\phi H\phi_x+H(\phi\phi_x)\big)v_k^+ ,\\
   & \\
   \tilde{Q}_k^{4,lin}(\phi, \phi, \phi, v) &= Q_k^{4,lin}(\phi, \phi, \phi, v)+\frac{3}{8}(i+2)\phi_{<k}^3v_k^+  -\frac{3}{8}iP_{<k}(\phi^3)v_k^+ + \frac{3}{4}(2i-1)\phi_{<k}^2\partial_xB_k^{lin}(\phi,v)\\
   &\qquad +\frac{3}{4}(\partial_x\phi_{<k})^2 B_k^{lin}(\phi, v) -\frac{3}{8}iP_{<k}\big(\phi H\phi_x + H(\phi \phi_x)\big)B_k^{lin}(\phi, v),\\
   & \\
   \tilde{Q}_k^{5,lin}(\phi, \phi, \phi, \phi,v) &= \frac{3}{8}(i+2)\phi_{<k}^3B_k^{lin}(\phi, v) -\frac{3}{8}iP_{<k}(\phi^3)B_k^{lin}(\phi, v) .
\end{aligned}
\end{equation*}
Clearly, $B_k^{lin}(\phi ,v)$ can be written as
\begin{equation}
  B_k^{lin}(\phi ,v) = 2^{-k}L(\phi_{<k}, v_k)+ \sum_{j\in (0,k)}2^{-j}L(\phi_k, v_j) + \sum_{j\geq k}2^{-j}L(\phi_j, v_j).   \label{BklinExpansion}
\end{equation}

Our goal is now to estimate the initial data for $w_k$ in $L^2$ and the inhomogeneous term in $L^1_tL_x^2$.
We begin with the initial data, for which we have
\begin{lemma}
The initial data $w_k(0)$ satisfies
\begin{equation}
    \| w_k^+(0)\|_{L^2}\lesssim 2^{\frac{k}{2}}d_k. \label{wk0Estimate}
\end{equation}
\end{lemma}
\begin{proof}
It suffices to prove the same estimate for $\tilde{v}_k^+$, which in turn reduced to establish the bound for $B_k^{lin}\left(\phi(0), v(0)\right)$.
For three components in \eqref{BklinExpansion}, 
\begin{gather*}
    \|2^{-k}L_k\left(\phi_{<k}(0),v_k(0)\right)  \|_{L^2}\leq 2^{-\frac{k}{2}}\|\phi(0) \|_{L^2}\|v_k(0) \|_{L^2}\lesssim \epsilon d_k, \\
    \|  \sum_{j \geq k}2^{-j}L_k\left(\phi_j(0), v_j(0)\right) \|_{L^2} \lesssim \sum_{j\geq k}2^{-\frac{j}{2}}\cdot \epsilon \|v_j(0) \|_{L^2}\lesssim \sum_{j\geq k} \epsilon\cdot d_j\lesssim \epsilon d_k, \\
    \| \sum_{j\in (0,k)}2^{-j}L(\phi_k, v_j)\|_{L^2_x} \lesssim \sum_{j\in (0,k)}2^{-\frac{j}{2}}\| \phi_k\|_{L_x^2}\|v_j \|_{L_x^2} \lesssim \epsilon d_k.
    \end{gather*}
Here we use the slow varying property of frequency envelope. 
\end{proof}

\par As is discussed in Section~\ref{s:NormalForm}, we consider the inhomogeneous term:
\begin{lemma}
The inhomogeneous terms in the $w_k^+$ equation \eqref{wkEqn} satisfy 
\begin{equation*}
    \|Q^{2,lin}_k \|_{L_t^1L_x^2}+ \|\tilde{Q}^{3,lin}_k \|_{L_t^1L_x^2}+\|\tilde{Q}^{4,lin}_k \|_{L_t^1L_x^2} +\|\tilde{Q}^{5,lin}_k \|_{L_t^1L_x^2}\lesssim 2^{\frac{k}{2}}C\epsilon d_k.
\end{equation*}
\end{lemma}
\begin{proof}
The $Q_k^{2,lin}$ term is easily estimated in $L^2_{t,x}$ using the bilinear estimates in the bootstrap assumption \eqref{AssumptionTwo}.
\begin{equation*}
  \|Q^{2,lin}_k \|_{L_{t,x}^2} \lesssim \epsilon C c_k\cdot 2^{-k+2k-\frac{k}{2}} \lesssim 2^{\frac{k}{2}}C\epsilon d_k.
\end{equation*}
Then for the term $B_k^{lin}(\phi, v)$, we only need to change a bit from the proof for initial data:
\begin{equation*}
   \|B_k^{lin}(\phi, v)\|_{\mathcal{S}} \lesssim C\epsilon d_k. 
\end{equation*}
\par For $\tilde{Q}_k^{3,lin}$ terms, we have similar estimate for $L_k(\Bar{\phi}, \Bar{\phi}, v)$, where $\Bar{\phi}$ is either $\phi$ or $H\phi$.
The only exception is the terms like $\partial_x^{-1}v_{(0,k)}P_k^+\partial_x(\phi H\phi_x)$, which admits the expansion
\begin{equation*}
 \partial_x^{-1}v_{(0,k)}P_k^+\partial_x(\phi H\phi_x) = \sum_{j\in\{0,k \}}2^{-j+2k}L_k(\Bar{\phi}_k ,\Bar{\phi}_{<k}, v_j)+ \sum_{j\in\{0,k \}}\sum_{l\geq k}2^{-j+2k}L_k(\Bar{\phi_l} ,\Bar{\phi_l}, v_j). 
\end{equation*}
Here we necessarily have two unbalanced frequencies, therefore this expression can be estimated by a bilinear $L^2$ estimate plus a Strichartz estimate.
\par The rest of proof is identical to the similar argument as lemma \ref{t:tildeQ345}.
\end{proof}
On account of the estimate for $w_k(0)$ \eqref{wk0Estimate}, the inhomogeneous terms and $B_k^{lin}(\phi, v)$, we prove the bound for $v_k^+$.
\par To recover the bilinear $L^2$ bounds we again use the bilinear Strichartz estimate,
\begin{equation*}
    \|\tilde{P}_jw_j^+ \cdot \tilde{P}_k\psi_k^+ \|_{L^2_{t,x}} \lesssim \epsilon d_j c_k 2^{\frac{j}{2}}2^{-max\{j,k\}}.
\end{equation*}
Without loss of generality, we assume that $j<k$.
To transform this bound to $v_j^+\phi_k^+$ we write
\begin{equation*}
\tilde{P}_jw_j^+ \tilde{P}_j\psi_k^+ - v_j^+e^{-i\Phi_{<j}}\phi_k^+e^{-i\Phi_{<k}} = \tilde{P}_jw_j^+(\tilde{P}_k \psi_k^+ - \phi_k^+e^{-i\Phi_{<k}})+(\tilde{P}_jw_j^+ - v_j^+ e^{-i\Phi_{<j}})\phi_k^+e^{-i\Phi_{<k}}.
\end{equation*}
For the first term, we use \eqref{FirstComponent} for the second factor and the bootstrap assumption \eqref{AssumptionOne} for the first factor. 
For the second term, we rewrite the first factor in a commutator structure as
\begin{equation*}
    \begin{aligned}
   \tilde{P}_jw_j^+ - v_j^+ e^{-i\Phi_{<j}} & = (\tilde{P}_j - 1)(v_j^+e^{-i\Phi_{<j}}) + B_j^{lin}(\phi, v)e^{-i\Phi_{<j}} \\
   &= [\tilde{P}_j-1 , e^{-i\Phi_{<j}}]v_j^+ +e^{-i\Phi_{<j}}(\tilde{P}_j - 1)B_j(\phi , v)+B_j^{lin}(\phi, v)e^{-i\Phi_{<j}} \\
   & = 2^{-j}L(\partial_xe^{-i\Phi_{<j}}, \phi_j^+)+ L\left(B_j^{lin}(\phi, v), e^{-i\Phi_{<j}}\right) \\
   & = 2^{-j}L(\phi_{<j},v_j,e^{-i\Phi_{<j}} ) +2^{-j}L(v_{<j},\phi_j,e^{-i\Phi_{<j}} ) +2^{-j}L(\partial_x^{-1}v_{(0,j)},\phi_j,e^{-i\Phi_{<j}} ) \\
   &\quad + \sum_{l>j}2^{-l}L(v_l, \phi_l ,e^{-i\Phi_{<j}}).
    \end{aligned}
\end{equation*}
Now we multiply this by $\phi_k^+$, discard the exponential factor and estimate in $L^2$ using our bootstrap hypothesis \eqref{AssumptionTwo}.
For $l\neq k$ we can use a bilinear $L^2$ estimate combined with an $L^\infty$ bound obtained via Bernstein's inequality.
For $l = k$, we estimated directly.
We obtain
\begin{equation*}
  \| (\tilde{P}_jw_j^+ - v_j^+ e^{-i\Phi_{<j}})\phi_k^+\|_{L^2_{t,x}} \lesssim C\epsilon^2 2^{-\frac{k}{2}}d_jd_k.  
\end{equation*}
Till now, we manage to prove \eqref{BootstrapOne} and \eqref{BootstrapTwo}, and this leads to \eqref{e:LinearizedBound}.

\section{$L^2$ well-posedness for third order Benjamin-Ono equation}\label{s:WellPosedness}
In this section, we show that the equation $\eqref{HBO}$ is locally well-posed in $L^2$, so that it becomes globally well-posed because of the conservation of mass. 
By scaling we can assume that our initial data satisfies
\begin{equation}
    \| \phi_0\|_{L^2} \leq \epsilon \ll 1. \label{InitialOneHalf}
\end{equation}
and prove well-posedness up to time $T = 1$. 
If the initial data is in $H^2$ then the solution exists(see \cite{MR2737850}).
And according to the bounds for linearized equation \eqref{e:LinearizedBound}, we can compare two solutions in $\mathcal{S}^{-\frac{1}{2}}$,
\begin{equation}
    \|\phi^{(1)} -\phi^{(2)} \|_{\mathcal{S}^{-\frac{1}{2}}} \lesssim \|\phi^{(1)}(0) -\phi^{(2)}(0) \|_{H^{-\frac{1}{2}}}. \label{Lipschitz}
\end{equation}
We call this property \textit{weak Lipschitz dependence on the initial data}(see \cite{MR3948114} and \cite{ifrim2020local}).

 Next we use this property to construct solution from smooth solutions as a limit for $L^2$ data.
 Given any initial data $\phi_0 \in L^2$ satisfying \eqref{InitialOneHalf}, we consider the corresponding regularized data
 \begin{equation*}
     \phi^{(n)} (0) = P_{<n}\phi_0. 
 \end{equation*}
 Then these satisfy
 \begin{equation*}
     \|\phi^{(n)} (0) \|_{L^2} \leq \epsilon, \qquad \|P_k \phi^{(n)}(0) \|_{L^2}\leq \epsilon c_k.
 \end{equation*}
 By virtue of Theorem \ref{Theorem4} , the corresponding solutions $\phi^{(n)}$ exist in time $[0,1]$, and satisfy the uniform bounds
 \begin{equation}
     \| P_k \phi^{(n)}\|_{\mathcal{S}} \lesssim \epsilon c_k. \label{DyadicBound}
 \end{equation}
 On the other hand, for their difference,
 \begin{equation*}
     \|\phi^{(n)} - \phi^{(m)} \|_{\mathcal{S}^{-\frac{1}{2}}} \lesssim \| \phi^{(n)}(0) - \phi^{(m)}(0)\|_{H^{-\frac{1}{2}}} \lesssim (2^{-n}+ 2^{-m})\epsilon.
 \end{equation*}
 Thus the sequence $\phi^{(n)}$ converges to some function $\phi$ in space $S^{-\frac{1}{2}}$, and
 \begin{equation*}
     \|\phi^{(n)} - \phi \|_{\mathcal{S}^{-\frac{1}{2}}} \lesssim 2^{-n}\epsilon.
 \end{equation*}
 In particular, for each dyadic component, we have the convergence in $S$. 
Therefore the function $\phi$ inherits the dyadic bounds \eqref{DyadicBound}
 \begin{equation}
     \| P_k \phi\|_{\mathcal{S}}\lesssim \epsilon c_k. \label{EnvelopBound}
 \end{equation}
 This further allows us to prove convergence in $l^2S$. 
 For fixed $k$, we write
 \begin{equation*}
    \lim \sup \|\phi^{(n)} - \phi \|_{l^2\mathcal{S}}\leq \lim \sup \|P_{<k}(\phi^{(n)} - \phi) \|_{l^2\mathcal{S}}+ \|P_{\geq k}\phi \|_{l^2\mathcal{S}}+\lim \sup \|P_{\geq k}\phi^{(n)} \|_{l^2\mathcal{S}}\lesssim c_{\geq k}.
 \end{equation*}
 Letting $k \rightarrow \infty$, we obtain
 \begin{equation*}
     \lim \|\phi^{(n)}-\phi \|_{l^2\mathcal{S}} = 0.
 \end{equation*}
 This property also implies uniform convergence in $C((0,1),L^2)$;  
 it in turn allows us to pass to the limit in the third order Benjamin-Ono equation, and prove that the limit $\phi$ solves the equation in the sense of distributions.
 \par Thus, for each initial data $\phi_0 \in L^2$ we have obtained a weak solution $\phi \in l^2\mathcal{S}$, as the limit of the solutions with regularized data. 
 Furthermore, this solution satisfies the frequency envelope bound \eqref{EnvelopBound}.
Now we consider the dependence of these weak solutions on the initial data. 
Suppose that a sequence of initial data $\phi^{(n)}(0)$ satisfying \eqref{InitialOneHalf} converges uniformly to $\phi_0$ in $L^2$, then by the weak Lipschitz dependence 
\begin{equation*}
    \phi^{(n)} \rightarrow \phi \quad \mbox{in}\quad \mathcal{S}^{-\frac{1}{2}}.
\end{equation*}
Hence for the corresponding solutions we estimate
\begin{equation*}
    \phi^{(n)} - \phi = P_{<k}(\phi^{(n)}-\phi) + P_{\geq k}\phi^{(n)} - P_{\geq k}\phi.
\end{equation*}
Here the first term on the right converges to zero in $l^2\mathcal{S}$ as $n\rightarrow \infty$ by the weak Lipschitz dependence \eqref{Lipschitz}, and the last term converges to zero by the frequency envelope bound \eqref{EnvelopBound}. 
Hence,
\begin{equation*}
    \limsup_{n\rightarrow \infty} \|\phi^{(n)} - \phi \|_{l^2\mathcal{S}} \leq \lim_{k\rightarrow \infty} \limsup_{n \rightarrow \infty} \| P_{\geq k}\phi^{(n)}\|_{l^2\mathcal{S}}.
\end{equation*}
It remains to show that this last right hand side vanishes. For this we use the frequency envelope bound \eqref{EnvelopBound} applied to $\phi^{(n)}$ as follows.
\par Given $\delta >0$, we have 
\begin{equation*}
    \|\phi^{(n)}(0) - \phi_0 \|_{L^2} \leq \delta, \qquad n\geq n_{\delta}.
\end{equation*}
Suppose $\epsilon c_k$ is an $L^2$ frequency envelop for $\phi_0$, then
\begin{equation*}
    \|P_{\geq k}\phi^{(n)} \|_{l^2\mathcal{S}} \lesssim \epsilon c_{\geq k} + \delta.
\end{equation*}
By taking the limsup and letting $k$ go to $\infty$,
\begin{equation*}
   \lim_{k \rightarrow \infty } \limsup_{n\rightarrow \infty} \|P_{\geq k}\phi^{(n)} \|_{l^2\mathcal{S}} \lesssim \delta.
\end{equation*}
Since $\delta$ is arbitrary, we have
\begin{equation*}
   \lim_{k \rightarrow \infty } \limsup_{n\rightarrow \infty} \|P_{\geq k}\phi^{(n)} \|_{l^2\mathcal{S}} = 0,
\end{equation*}
and the proof of the theorem is complete.

\section{Decay for the third order Benjamin-Ono equation} \label{s:Decay}
\par For the rest of this article we prove Theorem~\ref{t:MainTheoremTwo}, following the idea of \cite{ifrim2019dispersive}.
To show this, we first set up bootstrap assumptions
\begin{equation}
    |\phi(t,x)|\leq M\epsilon t^{-\frac{1}{4}}\left \langle x \right \rangle_{t}^{-\frac{1}{4}+\delta}, \quad |\phi_x(t,x)|\leq M\epsilon t^{-\frac{3}{4}}\left \langle x \right \rangle_{t}^{\frac{1}{4}+\delta}. \label{BootstrapThree}
\end{equation}  
And we further assume that 
\begin{equation}
   |\phi_{xx}(t,x)|\leq M\epsilon t^{-\frac{5}{4}}\left \langle x \right \rangle_{t}^{\frac{3}{4}+\delta}. \label{BootstrapFour}
\end{equation}
Then, we can also obtain a pointwise bound for $H\phi(x)$,
\begin{align*}
    |H\phi(x)| &\lesssim \left|\int_{|y-x|\leq c}\frac{\phi(y)}{x-y}dy\right| +  \left|\int_{|y-x|> c}\frac{\phi(y)}{|x-y|}dy\right|\\
   & \leq \int_{|y-x|\leq c}\frac{|\phi(y)-\phi(x)|}{|x-y|}dy + \int_{|y-x|> c} \frac{M\epsilon t^{-\frac{1}{4}}\left\langle y \right \rangle_{t}^{-\frac{1}{4}}}{|x-y|}dy \\
    &\lesssim  \int_{|y-x|\leq c} |\phi^{'}(x)|dy + \int_{|y-x|> c} M\epsilon t^{-\frac{1}{4}}\left\langle y \right \rangle_{t}^{-\frac{5}{4}} \lesssim M\epsilon  t^{-\frac{1}{4}}\left\langle x \right \rangle_{t}^{-\frac{1}{4}} .
\end{align*}
Given this set-up, our proof has three main steps:

\subsection{Uniform energy estimates for $\phi$}
Equation \eqref{HBO} is a completely integrable equation.
From \cite{ifrim2019dispersive} we have the uniform energy estimate
\begin{proposition}
Assume the solution $\phi$ of $\eqref{HBO}$ has initial data $\phi_0$ such that for every small positive constant $\delta$,
\begin{equation*}
    \| \phi_0\|_{B_{2,\infty}^{-\frac{1}{2}+ \delta}}\leq \epsilon.
\end{equation*}
Then 
\begin{equation*}
    \|\phi(t)\|_{B_{2,\infty}^{-\frac{1}{2}+ \delta}}\lesssim \epsilon.
\end{equation*} \label{t:BOReference}
\end{proposition}
In \cite{MR4270277}, this proposition is proved for Benjamin-Ono equation, and it also holds for Benjamin-Ono hierarchy. 
Note that this proposition is slightly different from Proposition $3.1$ of \cite{ifrim2019dispersive}, as the Besov index must be strictly greater than $-\frac{1}{2}$.

\subsection{$\dot{H}^{-\frac{1}{2}}$ bounds for the linearized equation}
\begin{proposition}
Let $\phi$ be a solution to the equation $\eqref{HBO}$ which satisfies the smallness of initial condition \eqref{SmallnessCondition}, as well as the bootstrap assumptions for $t \ll \epsilon^{-12}$. Then the solution of the linearized equation \eqref{LHBO} has almost conserved $\dot{H}^{-\frac{1}{2}}$ bound. 
\end{proposition}
This proposition is proved in next section.
We see that one solution for the linearized equation $\eqref{LHBO}$ is provided by
\begin{equation*}
    z : = \partial_x(L^{NL}\phi),
\end{equation*}
and $L^{NL}\phi$ \eqref{NonlinearL} solves the adjoint linearized equation
\begin{equation*}
    w_t - w_{xxx}+\frac{3}{4}w_x\phi^2 +\frac{3}{4}w_xH\phi_x +\frac{3}{4}(w_x\phi)_{x}+\frac{3}{4}Hw_{xx}\cdot\phi =0.
\end{equation*}
By the duality argument, the adjoint linearized equation has almost conserved $\dot{H}^{\frac{1}{2}}$ bounds.
Therefore we show that
\begin{proposition}
Let $\phi$ be a solution to the equation $\eqref{HBO}$ which satisfies the smallness of initial condition \eqref{SmallnessCondition}, as well as the bootstrap assumption \eqref{BootstrapThree}, \eqref{BootstrapFour} for $t \ll \epsilon^{-12}$. 
Then we have 
\begin{equation*}
    \|L^{NL}\phi(t)\|_{\dot{H}^{\frac{1}{2}}}\lesssim \epsilon, \quad t \ll \epsilon^{-12}.
\end{equation*}
\end{proposition}

\subsection{Nonlinear vector field Sobolev inequalities}
Taking into account the uniform bounds for $\phi$ and $L^{NL}\phi$, the pointwise bound in the theorem will follow for the following proposition:
\begin{proposition}
Let $t\ll \epsilon^{-12}$, $\delta >0$ be a small constant, and $\phi$ be a function satisfies the bootstrap assumptions \eqref{BootstrapThree}, as well as the bound
\begin{equation*}
    \|\phi\|_{\dot{B}^{-\frac{1}{2}+\delta}_{2,\infty}} + \| L^{NL}\phi\|_{\dot{H}^{\frac{1}{2}}}\lesssim \epsilon.
\end{equation*}
Then we have the pointwise bound
\begin{equation*}
     t^{\frac{1}{4}}\left \langle x \right \rangle_{t}^{\frac{1}{4}+\delta}|\phi(t,x)|+ t^{\frac{3}{4}}\left\langle x \right \rangle_{t}^{-\frac{1}{4}+\delta}|\phi_x(t,x)|\lesssim \epsilon.
\end{equation*}
In the elliptic region $E$ we have the additional bound
\begin{equation*}
   \left\langle x \right \rangle_{t} |\phi(t,x)| + t^{\frac{1}{2}}\left\langle x \right \rangle_{t}^{\frac{1}{2}}|\phi_x(t,x)|\lesssim \epsilon \log(t^{-\frac{1}{3}}\left\langle x \right \rangle_{t}).
\end{equation*} \label{t:vfBound}
\end{proposition}
This proposition is derived in the last section.

\section{Almost conserved quantity for linearized equation} \label{s:AlmostConserved}

In this section, we aim to show that for the linearized equation \eqref{LHBO}, we have the almost conserved $\dot{H}^{-\frac{1}{2}}$ norm, namely
\begin{equation*}
    \|v(t) \|_{\dot{H}^{-\frac{1}{2}}} \approx \| v(0)\|_{\dot{H}^{-\frac{1}{2}}}, \quad t\ll \epsilon^{-12}.
\end{equation*}
Setting $y = |D|^{-\frac{1}{2}}v$, note that $\partial_x \cdot |D|^{-\frac{1}{2}} = H |D|^{\frac{1}{2}}$, then the linearized equation becomes
\begin{equation}
(\partial_t - \partial_x^3)y + \frac{3}{4}H |D|^{\frac{1}{2}}(\phi^2 |D|^{\frac{1}{2}}y) + \frac{3}{4}H |D|^{\frac{1}{2}}(|D|^{\frac{1}{2}}y H\phi_x + \phi H |D|^{\frac{1}{2}} y_x)- \frac{3}{4} |D|^{\frac{1}{2}}\partial_x(|D|^{\frac{1}{2}}y \cdot \phi)=0. \label{EqnY}
\end{equation}
Define the energy $E^{[2]}(y) = \|y \|_{L^2}^2$, then for this equation \eqref{EqnY} we need to prove that its $L^2$ norm is almost conserved, namely,
\begin{equation}
    E^{[2]}(y(t))\approx E^{[2]}(y(0)), \quad t\ll \epsilon^{-12}. \label{AlmostY}
\end{equation}
We have
\begin{equation*}
\partial_t E^{[2]}(y) = \frac{3}{2}\int H |D|^{\frac{1}{2}} y\cdot \phi^2 |D|^{\frac{1}{2}}y  +  H |D|^{\frac{1}{2}}y(|D|^{\frac{1}{2}}y H\phi_x +\phi H |D|^{\frac{1}{2}} y_x) +  H |D|^{\frac{1}{2}}y H\partial_x(|D|^{\frac{1}{2}}y\cdot \phi)dx.
\end{equation*}
Now we consider the high-low interaction of each term of $\partial_t E^{[2]}(y)$, namely, we divide each term into $\lesssim t^{-\frac{1}{3}}$ and $\gg t^{-\frac{1}{3}}$ frequency pieces.
\par First, for $H |D|^{\frac{1}{2}} y\cdot \phi^2 |D|^{\frac{1}{2}}y$ term, we have 
\begin{enumerate}
    \item Three low frequencies:
    \begin{equation*}
    \left|\int H |D|^{\frac{1}{2}} y_{\lesssim t^{-\frac{1}{3}}}\cdot (\phi^2)_{\lesssim t^{-\frac{1}{3}}} |D|^{\frac{1}{2}}y_{\lesssim t^{-\frac{1}{3}}} dx\right|\lesssim t^{-\frac{1}{3}}\| \phi\|^2_{L^\infty}\|y\|^2_{L^2} \lesssim M^2\epsilon^2 t^{-\frac{5}{6}}\|y\|^2_{L^2}.
\end{equation*}
Here we use the Cauchy-Schwarz inequality together with above bootstrap assumption.
    \item Low frequency on $\phi$. 
     \begin{equation*}
   \int H |D|^{\frac{1}{2}} y_{\lesssim t^{-\frac{1}{3}}}\cdot (\phi^2)_{\gg t^{-\frac{1}{3}}} |D|^{\frac{1}{2}}y_{\lesssim t^{-\frac{1}{3}}} dx = 0,
\end{equation*}
as $H$ is skew-adjoint and we can commute it across $\phi^2$.
\item  Low frequency on either $y$ factor. 
Here we compute using some form of a fractional Leibniz rule distributing the first $|D|^{\frac{1}{2}}$ to the two other factors:
\begin{equation*}
    \left|\int H |D|^{\frac{1}{2}} y_{\gg t^{-\frac{1}{3}}}\cdot (\phi^2)_{\gg t^{-\frac{1}{3}}} |D|^{\frac{1}{2}}y_{\lesssim t^{-\frac{1}{3}}} dx\right|\lesssim t^{-\frac{1}{6}}\| |D|^{\frac{1}{2}}(\phi^2)_{\gg t^{-\frac{1}{3}}}\|_{L^\infty}\|y\|^2_{L^2} \lesssim M^2\epsilon^2 t^{-\frac{11}{12}}\|y\|^2_{L^2}.
\end{equation*}
Here for the $|D|^{\frac{1}{2}}(\phi^2)$ term we use the interpolation of $\phi$ and $\phi_x$ bound.
\end{enumerate}
\par Second, for $H |D|^{\frac{1}{2}}y \cdot |D|^{\frac{1}{2}}y \cdot H\phi_x$ term, we have
\begin{enumerate}
    \item Three low frequencies:
    \begin{equation*}
    \left|\int H |D|^{\frac{1}{2}} y_{\lesssim t^{-\frac{1}{3}}}\cdot (H\phi_x)_{\lesssim t^{-\frac{1}{3}}} |D|^{\frac{1}{2}}y_{\lesssim t^{-\frac{1}{3}}} dx\right|\lesssim t^{-\frac{1}{3}}\| H\phi_x\|_{L^\infty}\|y\|^2_{L^2} \lesssim M\epsilon t^{-\frac{13}{12}}\|y\|^2_{L^2}.
\end{equation*}
Here we use the Cauchy-Schwarz inequality together with bootstrap assumption \eqref{BootstrapThree}.
    \item Low frequency on $\phi$ vanishes as $H$ is skew-adjoint and we can commute it across $H\phi_x$.
\item  Low frequency on either $y$ factor.
Here we compute using some form of a fractional Leibniz rule distributing the first $|D|^{\frac{1}{2}}$ to the two other factors:
\begin{equation*}
    \left|\int H |D|^{\frac{1}{2}} y_{\gg t^{-\frac{1}{3}}}\cdot (H\phi_x)_{\gg t^{-\frac{1}{3}}} |D|^{\frac{1}{2}}y_{\lesssim t^{-\frac{1}{3}}} dx\right|\lesssim t^{-\frac{1}{6}}\| |D|^{\frac{1}{2}}(H\phi_x)_{\gg t^{-\frac{1}{3}}}\|_{L^\infty}\|y\|^2_{L^2} \lesssim M\epsilon t^{-\frac{7}{6}}\|y\|^2_{L^2}.
\end{equation*}
Here for the $|D|^{\frac{1}{2}}(H\phi_x)$ term we use the interpolation of $\phi_{xx}$ and $\phi_x$ bound.
\end{enumerate}

\par Next, for $H|D|^{\frac{1}{2}}y\cdot H|D|^{\frac{1}{2}}y_x\cdot \phi$ term, using the integration by parts, it is equivalent to consider $(H|D|^{\frac{1}{2}}y)^2 \phi_x$, there are two cases:
\begin{enumerate}
    \item Two terms on $y$ both have low frequencies,
    \begin{equation*}
    \left|\int (H |D|^{\frac{1}{2}} y_{\lesssim t^{-\frac{1}{3}}})^2\cdot (\phi_x)_{\lesssim t^{-\frac{1}{3}}}  dx\right|\lesssim t^{-\frac{1}{3}}\| H\phi_x\|_{L^\infty}\|y\|^2_{L^2} \lesssim M\epsilon t^{-\frac{13}{12}}\|y\|^2_{L^2}.
    \end{equation*}
    \item One $y$ has high frequency and the other has low frequency,
      \begin{equation*}
    \left|\int H |D|^{\frac{1}{2}} y_{\lesssim t^{-\frac{1}{3}}}  H |D|^{\frac{1}{2}} y_{\gg t^{-\frac{1}{3}}}\cdot (\phi_x)_{\lesssim t^{-\frac{1}{3}}}  dx\right|\lesssim t^{-\frac{1}{6}}\| H|D|^{\frac{1}{2}}\phi_x\|_{L^\infty}\|y\|^2_{L^2} \lesssim M\epsilon t^{-\frac{7}{6}}\|y\|^2_{L^2}.
    \end{equation*}
\end{enumerate}
\par Last, the term $H|D|^{\frac{1}{2}}y \cdot H\partial_x(|D|^{\frac{1}{2}}y\cdot \phi)$ is equivalent to $|D|^{\frac{1}{2}}y \cdot \partial_x(|D|^{\frac{1}{2}}y\cdot \phi)$ in the integral, and it is similar to above cases.

\par To summarize, we have showed that we only have to consider all the high frequency terms in $\partial_t E^{[2]}(y)$, and the remaining terms are just $O(M\epsilon t^{-\frac{11}{12}}\|y\|_{L^2}^2)$.
The expression is too large to be estimated directly in terms of $\|y\|_{L^2}^2$ energy.
We instead use a  normal form energy correction to eliminate it (\cite{MR3348783}).
Precisely, we will seek to eliminate (the bulk of) this expression by adding a cubic correction to the quadratic energy functional, at the expense of producing further quartic errors; these quartic errors will be bounded.
Here we seek the normal form of the type $\tilde{y} = y+B(y,\phi)$ to eliminate the quadratic terms in \eqref{EqnY}, the computation shows that
\begin{equation*}
   \tilde{y} = y+\frac{1}{4}|D|^{-\frac{1}{2}}(H|D|^{\frac{1}{2}}y\cdot \partial_x^{-1}\phi -H|D|^{-\frac{1}{2}}y\cdot H\phi) +\frac{1}{4} |D|^{\frac{1}{2}}(H|D|^{-\frac{1}{2}}y\partial_x^{-1}\phi).
\end{equation*}
Hence, formally, the correction of the energy would be
\begin{equation*}
    E^{(3)}= \frac{1}{8}\int |D|^{-\frac{1}{2}}y(H|D|^{\frac{1}{2}}y\cdot \partial_x^{-1}\phi -H|D|^{-\frac{1}{2}}y\cdot H\phi) + |D|^{\frac{1}{2}}y \cdot H|D|^{-\frac{1}{2}}y\cdot\partial_x^{-1}\phi dx.
\end{equation*}
Then we define the modified energy as
\begin{equation*}
    E := E^{[2]}+E^{[3]},
\end{equation*}
where $y^{hi}: = y_{\gg t^{-\frac{1}{3}}}$, $\phi^{hi}: = \phi_{\gg t^{-\frac{1}{3}}}$ and
\begin{equation*}
     E^{[3]}= \frac{1}{8}\int |D|^{-\frac{1}{2}}y^{hi}(H|D|^{\frac{1}{2}}y^{hi}\cdot \partial_x^{-1}\phi^{hi} -H|D|^{-\frac{1}{2}}y^{hi}\cdot H\phi^{hi}) + |D|^{\frac{1}{2}}y^{hi} \cdot H|D|^{-\frac{1}{2}}y^{hi}\cdot\partial_x^{-1}\phi^{hi} dx.
\end{equation*}
We need to prove the norm equivalence,
\begin{equation*}
    E(y) \approx \|y\|^2_{L^2},
\end{equation*}
and the slow growth
\begin{equation}
    \partial_t E(y) \lesssim M\epsilon t^{-\frac{11}{12}}\|y\|^2_{L^2}. \label{SlowGrowth}
\end{equation}
For the first bound, 
\begin{equation*}
    |E^{[3]}(y)|\lesssim \| \partial_x^{-1}\phi\|_{L^\infty}\|y\|_{L^2}^2 + t^{\frac{1}{3}}\| \phi\|_{L^\infty}\|y\|_{L^2}^2 \lesssim \epsilon t^{\frac{1}{12}}\|y\|_{L^2}^2,
\end{equation*}
which suffices as $t\epsilon^{12}\lesssim 1$.
For the bound of growth, we compute,
\begin{equation*}
  \partial_t E(y) = D_1+D_2+D_3 +O(M\epsilon t^{-\frac{11}{12}}\|y \|_{L^2}^2).  
\end{equation*}
\begin{enumerate}
    \item $D_1$ arises from the scale change in the truncation, as the multiplier $P^{hi}$ is time dependent, with symbol of the form
    \begin{equation*}
        p^{hi}(\xi) := \xi(t^{\frac{1}{3}}\xi).
    \end{equation*}
    Its time derivative has the form
    \begin{equation*}
        \partial_t p^{hi}(\xi) = t^{-1}\frac{t^{\frac{1}{3}}\xi}{3}\xi^{'}(t^{\frac{1}{3}}\xi),
    \end{equation*}
    which is supported exactly in the region $|\xi|\approx t^{-\frac{1}{3}}$ , and we harmlessly abbreviate it as $t^{-1}P_{t^{-\frac{1}{3}}}$.
    Then the corresponding error term is
\begin{equation*}
\begin{aligned}
     D_1 \approx t^{-1}\int |D|^{-\frac{1}{2}}y_{t^{\frac{1}{3}}}\cdot H|D|^{\frac{1}{2}}y^{hi}\cdot \partial_x^{-1}\phi^{hi}+ |D|^{-\frac{1}{2}}y^{hi}\cdot H|D|^{\frac{1}{2}}y^{hi}\cdot \partial_x^{-1}\phi_{t^{\frac{1}{3}}}    \\
     + |D|^{-\frac{1}{2}}y_{t^{\frac{1}{3}}}\cdot H|D|^{-\frac{1}{2}}y^{hi}\cdot H\phi^{hi}  +|D|^{-\frac{1}{2}}y^{hi}\cdot H|D|^{-\frac{1}{2}}y^{hi}\cdot H\phi_{t^{\frac{1}{3}}}  dx.
\end{aligned}
\end{equation*}
    \item $D_2$ is the term arising from $\phi_t$,
    \begin{equation*}
    \begin{aligned}
        D_2 \approx \int |D|^{-\frac{1}{2}}y^{hi}\cdot H|D|^{\frac{1}{2}}y^{hi}\cdot \left(\phi^3+3\phi H\phi_x +3H(\phi \phi_x)\right)^{hi} + \\
        |D|^{-\frac{1}{2}}y^{hi}\cdot H|D|^{-\frac{1}{2}}y^{hi}\cdot H\partial_x\left(\phi^3+3\phi H\phi_x +3H(\phi \phi_x)\right)^{hi} dx.
        \end{aligned}
    \end{equation*}
    \item $D_3$ is the term arising from $y_t$,
     \begin{equation*}
    \begin{aligned}
  D_3 \approx \int\left(\partial_x (|D|^{\frac{1}{2}}y\cdot \phi)-H(\phi^2|D|^{\frac{1}{2}}y)-H(|D|^{\frac{1}{2}}y H\phi_x +\phi H|D|^{\frac{1}{2}}y_x)\right)^{hi}\cdot H|D|^{\frac{1}{2}}y^{hi}\cdot \partial_x^{-1}\phi^{hi}\\
  +\left(|D|(\phi^2|D|^{\frac{1}{2}}y)+|D|(|D|^{\frac{1}{2}}yH\phi_x +\phi H|D|^{\frac{1}{2}}y_x)-\partial_x^2(|D|^{\frac{1}{2}}y\cdot \phi)\right)^{hi}\cdot |D|^{-\frac{1}{2}} y^{hi}\cdot \partial_x^{-1}\phi^{hi}\\
  +\left(\partial_x (|D|^{\frac{1}{2}}y\cdot \phi)-H(\phi^2|D|^{\frac{1}{2}}y)-H(|D|^{\frac{1}{2}}yH\phi_x +\phi H|D|^{\frac{1}{2}}y_x)\right)^{hi}\cdot H|D|^{-\frac{1}{2}}y^{hi}\cdot H\phi^{hi} dx.
     \end{aligned}
    \end{equation*}
\end{enumerate}
For $D_1$, we use the some form of a fractional Leibniz rule to  distribute the multiplier $|D|^{\frac{1}{2}}$ to other factors and the pointwise bound
\begin{equation*}
    |\partial_x^{-1}\phi^{hi}| + |\partial_x^{-1}\phi_{t^{\frac{1}{3}}}|\lesssim M\epsilon
\end{equation*}
to deduce 
\begin{equation*}
    |D_1|\lesssim M\epsilon t^{-\frac{11}{12}}\|y\|_{L^2}^2.
\end{equation*}
For $D_2$ we use the bootstrap assumption for $\phi$ and $\phi_x$ to deduce that
\begin{equation*}
    |D_2|\lesssim M^2\epsilon^2 t^{-1}\|y\|_{L^2}^2,
\end{equation*}
which again suffices.\\
Finally for $D_3$ we  redistribute the Littlewood-Paley operator $P^{hi}$ to other non-$y_t$ terms.
We then distribute the multiplier $|D|$, $|D|^{\frac{1}{2}}$ to each of the other factors using the fractional Leibniz rule.
Using the bootstrap bounds, 
\begin{equation*}
    |D_3|\lesssim M^2\epsilon^2 t^{-1}\|y\|_{L^2}^2.
\end{equation*}
Thus the \eqref{SlowGrowth} is proved.
Applying Gronwall’s inequality for the modified energy functional $E(y)$, we obtain \eqref{AlmostY} .

\section{Nonlinear vector field Sobolev  estimates} \label{s:nonlinear vf}
In this section we prove Proposition \ref{t:vfBound}, and conclude the proof of decay estimates.
For the result stating below,  it will be convenient to rescale the function so that we can eliminate the $t$ in $L^{NL}$. 
For the equation
\begin{equation*}
    x\phi + 3t\phi_{xx} - \frac{3t}{4}\phi^3 - \frac{3t}{4}[\phi H\phi_x +H(\phi\phi_x)] = f, \quad f: = L^{NL}\phi,
\end{equation*}
we define the rescaled function
\begin{equation*}
    \tilde{\phi}(x) := t^{\frac{1}{3}}\phi(t, xt^{\frac{1}{3}}),\quad \tilde{f}(x) := f(t, xt^{\frac{1}{3}})
\end{equation*}
solves the equation when $t=1$,
\begin{equation*}
    x\tilde{\phi} + 3\tilde{\phi}_{xx} - \frac{3}{4}\tilde{\phi}^3 - \frac{3}{4}[\tilde{\phi} H\tilde{\phi}_x +H(\tilde{\phi}\tilde{\phi}_x)] = \tilde{f}.
\end{equation*}
Also the energy bound for $\phi$ becomes
\begin{equation}
    \| \tilde{\phi}\|_{B^{-\frac{1}{2}+\delta}_{2, \infty}} \lesssim \tilde{\epsilon}: =\epsilon t^{\frac{\delta}{3}}\ll 1 . \label{BesovBound}
\end{equation}
Let $\left\langle x \right \rangle : = \left\langle x \right \rangle_1$ be the usual Japanese bracket, $R$ be any dyadic number and  the region $A_R : = \{ \left\langle x \right \rangle\approx R \gtrsim 1\}$, and define $\chi_R$ to be a positive weight function such that that $\chi_R \approx 1$ on $A_R$, decays  to zero away from $A_R$ and $|\chi^{'}_R|\leq R^{-\frac{1}{4}}$.

 \begin{lemma}
 With the rescaled bootstrap assumptions
 \begin{equation}
     |\tilde{\phi}(x)|\leq M\tilde\epsilon \left \langle x \right \rangle^{-\frac{1}{4}+\delta}, \quad |\tilde{\phi}_x(x)|\leq M\tilde\epsilon \left \langle x \right \rangle^{\frac{1}{4}+\delta}, \label{BootstrapFive}
 \end{equation}
the uniform bound for $L^{NL}\tilde{\phi}$
\begin{equation}
    \|L^{NL} \tilde{\phi}\|_{\dot{H}^{\frac{1}{2}}}\lesssim \epsilon, \label{LNLBound}
\end{equation} 
as well as the Besov energy bound \eqref{BesovBound}, we have that 
\begin{equation}
    \|\tilde{\phi}\|_{L^2(A_R)}\lesssim \tilde{\epsilon} R^{\frac{1}{4}-\frac{\delta}{2}}, \quad 
    \|\tilde{\phi}_x\|_{L^2(A_R)}\lesssim \tilde{\epsilon} R^{\frac{3}{4}-\frac{\delta}{2}}, \quad
    \|\tilde{\phi}_{xx}\|_{L^2(A_R)}\lesssim \tilde{\epsilon} R^{\frac{5}{4}-\frac{\delta}{2}}. \label{L2ARBound}
\end{equation}
\end{lemma}
\begin{proof}
The low frequency bound follows directly from above Besov energy bound \eqref{BesovBound} for $\tilde{\phi}$, it suffices to consider high frequencies $\lambda \geq R^{\frac{1}{2}}$.
We first consider the bound for $L\tilde{\phi}_\lambda$, we have
\begin{equation*}
    L\tilde{\phi}_\lambda = P_\lambda L^{NL}\tilde{\phi} + \lambda^{-1}\tilde{\phi}_{\lambda} + P_\lambda\left(\tilde{\phi}^3 + \tilde{\phi}H\tilde{\phi}_x + H(\tilde{\phi} \tilde{\phi_x})\right).
\end{equation*}
Using the Besov energy bound \eqref{BesovBound} and the uniform energy bound \eqref{LNLBound} we have 
\begin{equation*}
    \|P_\lambda L^{NL}\tilde{\phi} \|_{L^2}\lesssim \epsilon\lambda^{-\frac{1}{2}}, \quad \| \tilde{\phi}_\lambda\|_{L^2}\lesssim \tilde{\epsilon}\lambda^{\frac{1}{2}-\delta}.
\end{equation*}
For the third part, we apply the pointwise bootstrap assumptions \eqref{BootstrapFive} on $\tilde{\phi}$ and $\tilde{\phi}_x$,
\begin{equation*}
    \|P_\lambda(\tilde{\phi}^3) \|_{L^2(A_R)}\leq M^3\tilde{\epsilon}^3R^{-\frac{1}{4}}, \quad \| P_\lambda(\tilde{\phi}H\tilde{\phi}_x + H(\tilde{\phi} \tilde{\phi_x}))\|_{L^2(A_R)} \leq M^2\tilde{\epsilon}^2 R^{\frac{1}{2}}.
\end{equation*}

Therefore we have 
\begin{equation*}
    \| L\tilde{\phi}_\lambda \|_{L^2(A_R)}\lesssim \tilde{\epsilon} R^{\frac{1}{2}}.
\end{equation*}
Also we see that \eqref{BesovBound} gives
\begin{equation*}
    \| \tilde{\phi}_\lambda \|_{L^2(A_R)}\lesssim \tilde{\epsilon} \lambda^{\frac{1}{2}-\delta}.
\end{equation*}
To obtain the bound on the derivative, we integrate by parts in $A_R$ to get
\begin{equation*}
    \int \chi_R |\tilde{\phi}_{\lambda,x}|^2 dx = \int \frac{1}{3}\chi_R \tilde{\phi}_\lambda L\tilde{\phi}_\lambda + (\frac{1}{2}\chi_R^{''}- \frac{1}{3}x\chi_R)|\tilde{\phi}_\lambda|^2 dx.
\end{equation*}
By H{\"o}lder's inequality this yields the bound
\begin{equation*}
    \|\tilde{\phi}_{\lambda, x} \|^2_{L^2(A_R)}\lesssim \tilde{\epsilon}^2 R\lambda^{\frac{1}{2}-\delta}.
\end{equation*}
Now if we consider the antiderivative localized at frequency $\lambda$, and harmlessly denote this operator by $\partial^{-1}_{x,\lambda}$, then we can express $\tilde{\lambda}_{x}$ in terms of $\tilde{\phi}_{\lambda,x}$,
\begin{equation*}
    \chi_R\tilde{\phi}_\lambda = \chi\partial^{-1}_{x,\lambda}\tilde{\phi}_{\lambda,x} =\partial^{-1}_{x,\lambda}(\chi_R\tilde{\phi}_{\lambda,x}) - [\partial^{-1}_{x,\lambda}, \chi_R]\tilde{\phi}_{\lambda,x}.
\end{equation*}
By the commutator estimates, we have the local bound for $\tilde{\phi}_\lambda$,
\begin{equation*}
    \| \tilde{\phi}_\lambda\|_{L^2(A_R)} \lesssim \tilde{\epsilon} R^{\frac{1}{2}}\lambda^{-\frac{1}{2}-\delta}.
\end{equation*}
Repeating the argument above we obtain the refined bounds
\begin{equation*}
  \|\tilde{\phi}_{\lambda,x}\|_{L^2(A_R)} \lesssim \tilde{\epsilon} R\lambda^{-\frac{1}{2}-\delta}, \quad \| \tilde{\phi}_\lambda\|_{L^2(A_R)} \lesssim \tilde{\epsilon} R\lambda^{-\frac{3}{2}-\delta},
\end{equation*}
and furthermore, 
\begin{equation*}
     \|\tilde{\phi}_{\lambda,xx}\|_{L^2(A_R)} \lesssim \tilde{\epsilon} R\lambda^{\frac{1}{2}-\delta},
\end{equation*}
which completes the proof.
\end{proof}

\begin{lemma}
With the bootstrap assumptions \eqref{BootstrapFive}, the uniform bound for $L^{NL}\tilde{\phi}$ \eqref{LNLBound} as well as the Besov energy bound \eqref{BesovBound}, we have
\begin{equation*}
    \|L^{NL}\tilde{\phi} \|_{L^2(A_R)}\lesssim \tilde{\epsilon} R^{\frac{1}{2}+\delta}.
\end{equation*}
\end{lemma}
\begin{proof}
We consider the frequency cutoff $R^{-1}$, then $L^{NL}\tilde{\phi}$ has the decomposition
\begin{equation*}
    L^{NL}\tilde{\phi} = P_{\gtrsim R^{-1}} L^{NL}\tilde{\phi} + P_{<R^{-1}}L\tilde{\phi} + P_{<R^{-1}}\left(\tilde{\phi}^3 + \tilde{\phi}H\tilde{\phi}_x + H(\phi \tilde{\phi_x})\right).
\end{equation*}
For the first term, we use the uniform $\dot{H}^{\frac{1}{2}}$ bound for $L^{NL}\tilde{\phi}$.
For the second term, we have
\begin{equation*}
P_{<R^{-1}} L\tilde{\phi} = L\tilde{\phi}_{<R^{-1}} + [P_{<R^{-1}}, L]\tilde{\phi} = L\tilde{\phi}_{<R^{-1}} +  [L, P_{\gtrsim R^{-1}}]\tilde{\phi},
\end{equation*}
because $L$ is a linear operator.
By the Besov energy bound \eqref{BesovBound} on $\tilde{\phi}$, 
\begin{equation*}
    \|L\tilde{\phi}_{<R^{-1}}\|_{L^2(A_R)} \lesssim R\| \tilde{\phi}\|_{L^2(A_R)} \lesssim \tilde{\epsilon} R^{\frac{1}{2}+\delta}.
\end{equation*}
The Fourier multiplier of the commutator of size $R$ is supported near $|\xi| \sim R^{-1}$, we have
\begin{equation*}
    \|[L, P_{\gtrsim R^{-1}}]\tilde{\phi}\|_{L^2(A_R)} \lesssim \|R^{-1}\tilde{\phi} \|_{L^2(A_R)} \lesssim \tilde{\epsilon} R^{-\frac{1}{2}-\delta}.
\end{equation*}
For the third term, we use the bound \eqref{L2ARBound}.
\end{proof}

Now we consider the problem in dyadic region $\{ |x|\approx R\}$, where $R \gtrsim 1$.
Setting $v = \chi_R \tilde{\phi}$, and $g: = \chi_R \tilde{f}$, it follows that $v$ solves the equation 
\begin{equation*}
    (x + 3\partial_x^2)v -\frac{3}{4}\tilde{\phi}^2v - \frac{3}{4}[vH\tilde{\phi}_x + H(\tilde{\phi}\tilde{\phi}_x)\chi_R] = g,
\end{equation*}
where $v$ and $f$ satisfy the bounds
\begin{equation}
    \|v\|_{L^2(A_R)}\lesssim \tilde{\epsilon} R^{\frac{1}{4}-\frac{\delta}{2}}, \quad  \|v_x\|_{L^2(A_R)}\lesssim \tilde{\epsilon} R^{\frac{3}{4}-\frac{\delta}{2}}, \quad
    \|v_{xx}\|_{L^2(A_R)}\lesssim \tilde{\epsilon} R^{\frac{5}{4}-\frac{\delta}{2}}, \label{BoundForV}
\end{equation}
respectively
\begin{equation}
    \|g\|_{\dot{H}^{\frac{1}{2}}}\lesssim \epsilon,\quad  \|g \|_{L^2} \lesssim \tilde{\epsilon} R^{\frac{1}{2}+\delta}. \label{FBound}
\end{equation}
As before, we split the frequency at $R^{\frac{1}{2}}$ to high and low frequencies.
\par \textbf{1.Pointwise estimate in the hyperbolic region: $x\approx R \gg 1$.}\\
Here we consider the region $A^H_R$ to the right of the origin, 
and apply hyperbolic energy estimates to establish the desired pointwise bound for $u$ supported in  $A^H_R$.
We use the
frequency scale $R^{\frac{1}{2}}$ to split 
\begin{equation*}
    g = g_{hi} + g_{lo},
\end{equation*}
then we have the energy conservation type relation
\begin{equation}
    \frac{d}{dx}\left(x|v|^2 +3|v_x|^2 - 2f_{lo} v\right) = |v|^2 +2g_{hi}v_x - 2g_{lo,x}v+ \frac{3}{2}\tilde{\phi}^2vv_x + \frac{3}{2}vv_xH\tilde{\phi}_x +\frac{3}{2}H(\tilde{\phi}\tilde{\phi}_x)\chi_Rv_x. \label{EnergyConservation}
\end{equation}
The first nonlinear term on the right hand side can be written in the form
\begin{equation*}
    \frac{3}{2}\tilde{\phi}^2vv_x = \frac{3}{8}\frac{d}{dx}(\tilde{\phi}^4\chi_R^2)+\frac{3}{4}\tilde{\phi}^4\chi_R\chi_R^{'}.
\end{equation*}
Thus the first term on the right hand side of above equality can be added to the left hand side of \eqref{EnergyConservation} as the energy correction,
\begin{equation*}
     \frac{d}{dx}\left(x|v|^2 +3|v_x|^2 - 2g_{lo} v - \frac{3}{8}\tilde{\phi}^4\chi_R^2 \right) = |v|^2 +2g_{hi}v_x - 2g_{lo,x}v+ \frac{3}{4}\tilde{\phi}^4\chi_R\chi_R^{'} + \frac{3}{2}vv_xH\tilde{\phi}_x +\frac{3}{2}H(\tilde{\phi}\tilde{\phi}_x)\chi_Rv_x.
\end{equation*}
Then applying Gronwall's inequality we obtain
\begin{equation*}
    \sup_{x\in A^H_R}\{x|v|^2+ 3|v_x|^2 \} \lesssim \tilde{\epsilon}^2R^{\frac{1}{2}-\delta}+   \sup_{x\in A^H_R}\{g_{lo}v \} + \sup_{x_0\in A^H_R}\int_{-\infty}^{x_0} R^{-1}\tilde{\phi}^4\chi_R + vv_xH\tilde{\phi}_x+ H(\tilde{\phi}\tilde{\phi}_x)\chi_Rv_x dx.
\end{equation*}
Now we need to estimate the right hand side of above inequality.
We first estimate the $L^\infty$ norm of $f_{lo}$.
By the second inequality of \eqref{FBound}, $\| g_{< R^{-1}}\|_{L^2}\lesssim \tilde{\epsilon} R^{\frac{1}{2}+\delta}$. 
Applying Bernstein’s inequality,  
\begin{equation*}
    \| g_{< R^{-1}}\|_{L^\infty}\lesssim  R^{-\frac{1}{2}}\| g_{< R^{-1}}\|_{L^2}\lesssim \tilde{\epsilon}R^{\delta}. 
\end{equation*}
For $R^{-1}\leq\lambda \leq R^{\frac{1}{2}}$, $\lambda$ dyadic, there are $O(\log R)$ number of pieces. 
We use the first inequality of \eqref{FBound}, $\| g_{\lambda}\|_{L^2}\lesssim \epsilon \lambda^{-\frac{1}{2}}$, and again apply the Bernstein’s inequality,
\begin{equation*}
    \| g_{\lambda}\|_{L^\infty}\lesssim  \lambda^{\frac{1}{2}}\| g_{\lambda}\|_{L^2}\lesssim \epsilon.
\end{equation*}
Summing up these frequencies, 
\begin{equation*}
    \| g_{lo}\|_{L^\infty} \lesssim \| g_{<R^{-1}}\|_{L^\infty} + \sum_{\lambda \geq R^{-1}}^{R^{\frac{1}{2}}} \|g_{\lambda}\|_{L^\infty} \lesssim \tilde{\epsilon} R^\delta.
\end{equation*}

\par Next, for the integral term, we use \eqref{BoundForV} the bootstrap assumption \eqref{BootstrapFive} to obtain
\begin{equation*}
    \int_{-\infty}^{x_0} R^{-1}\tilde{\phi}^4\chi_R dx \leq \| \tilde{\phi}\|^4_{L^\infty(A_R)}\int R^{-1}\chi_R dx \lesssim M^4 \tilde{\epsilon}^4 R^{-1+4\delta}.
\end{equation*}

Since $\chi_R v_x = \chi_R(\tilde{\phi}\chi_R)_x = \tilde{\phi}_x\chi_R^2 + \tilde \phi \chi^{'}_R \chi_R$, and $|\chi^{'}_R|\leq R^{-\frac{1}{4}}$,
\begin{equation*}
    \int_{-\infty}^{x_0} vv_xH\tilde{\phi}_x+H(\tilde{\phi}\tilde{\phi}_x)\chi_Rv_x dx = \int_{-\infty}^{x_0}  \tilde{\phi}\tilde{\phi}_x\chi^2_R H\tilde{\phi}_x + \tilde{\phi}_x \chi_R^2 H(\tilde{\phi}\tilde{\phi}_x) dx+  \int_{-\infty}^{x_0} H(\tilde{\phi}\tilde{\phi}_x)\chi_R^{'}v  + \tilde{\phi}_xH(v\tilde{\phi}\chi_R^{'} )dx.
\end{equation*}
For the second term on the right,
\begin{equation*}
     \int_{-\infty}^{x_0} H(\tilde{\phi}\tilde{\phi}_x)\chi_R^{'}v  + \tilde{\phi}_xH(v\tilde{\phi}\chi_R^{'} )dx \lesssim M^3\tilde{\epsilon}^3R^{\frac{1}{2}+\frac{3}{2}\delta}.
\end{equation*}
It suffices to consider the first term of the right hand side.
Note that using the definition of Hilbert transform, this integral can be rewritten as
\begin{equation*}
    PV\frac{1}{\pi}\int_{-\infty}^{x_0}\int^{\infty}_{-\infty} \left(\chi_R^2(x)-\chi_R^2(y)\right)\frac{\tilde{\phi}(x) \tilde{\phi}_x(x)\tilde{\phi}_x(y)}{x-y}dydx.
\end{equation*}
We consider two cases:
\begin{enumerate}
    \item If $|x-y|\lesssim R^{\frac{1}{4}-\delta}$, we apply the pointwise bound \eqref{BootstrapFive} for $\tilde{\phi}(x) \tilde{\phi}_x(x)\tilde{\phi}_x(y)$, and 
    \begin{equation*}
    \left|PV\int_{-\infty}^{x_0}\int^{\infty}_{-\infty} \frac{1}{x-y}\mathbf{1}_{(|x-y|\lesssim R^{\frac{1}{4}})}dydx\right| \leq \int^{\infty}_{-\infty}  \int^{\infty}_{-\infty} \frac{1}{|x-y|}\mathbf{1}_{(|x-y|\lesssim R^{\frac{1}{4}})}dydx.   
    \end{equation*}
    This yields a bound of $M^3\tilde{\epsilon}^3 R^{\frac{1}{2}+2\delta}$.
    \item If $|x-y|> R^{\frac{1}{4}-\delta}$, then we use \eqref{BoundForV} to bound $\chi_R\tilde{\phi}_x(x)$ and $\chi_R\tilde{\phi}_x(y)$ in $L^2$.
    Denoting $x_1 = \min\{x_0, y-t\}$, we integrate by parts to get
    \begin{equation*}
        \int_{-\infty}^{x_1} \frac{1}{x-y}\tilde{\phi}(x)dy = \frac{1}{x_1 -y}\phi(x_1)- \int_{-\infty}^{x_1} \frac{1}{(x -y)^2}\phi(x)dx.
    \end{equation*}
    As $|x_1 -y|>R^{\frac{1}{4}-\delta}$ from \eqref{BootstrapFive} we obtain a bound of  $M^3\tilde{\epsilon}^3 R^{\frac{1}{2}+\delta}$.
\end{enumerate}

Gathering all the terms above, we have the desired bound
\begin{equation*}
    \sup_{x\in A^H_R}\{x|v|^2+ 3|v_x|^2 \} \lesssim \tilde{\epsilon}^2 R^{\frac{1}{2}+2\delta}.
\end{equation*}
\par \textbf{2.Pointwise estimate in the self-similar region: $|x|\lesssim R = 1$.}\\
Here we simply use Sobolev embedding $\dot{H}^{\frac{1}{2}+\epsilon}(\mathbb{R})\subset L^\infty(\mathbb{R})$.
From the interpolation of \eqref{BoundForV}, we know that 
\begin{equation*}
    \| v\|_{L^\infty}\lesssim \| v\|_{\dot{H}^{\frac{1}{2}+\epsilon}(A_R)}\lesssim  \tilde{\epsilon} R^{\frac{1}{2}+\epsilon^{'}} \leq \tilde{\epsilon} R^{-\frac{1}{4}+\delta}.
\end{equation*}
\par \textbf{3.Pointwise estimate in the elliptic region: : $-x\approx R \gg 1$.}\\
In the elliptic region, the leading part of $v$ is $x^{-1}g_{lo}$.
As in the hyperbolic region, we split function $g$ into high and low frequency parts from frequency $R^{\frac{1}{2}}$
Subtracting this leading term, we are left with
\begin{equation*}
    v_1 : = v-x^{-1}g_{lo},
\end{equation*}
which solves
\begin{equation*}
    Lv_1 = g_1 := g_{hi} -3(x^{-1}g_{lo})_{xx}+\frac{3}{4}\tilde{\phi}^2v + \frac{3}{4}[vH\tilde{\phi}_x + H(\tilde{\phi}\tilde{\phi}_x)\chi_R].
\end{equation*}
$g_{hi}$ serves the leading part of $g_1$, while the other terms are perturbative, and by \eqref{FBound}
\begin{equation}
    \|g_1\|_{L^2}\lesssim \epsilon R^{-\frac{1}{4}}.\label{FOneBound}
\end{equation}
Multiplying above equation by $v_1$ and integrating by parts, 
\begin{equation*}
    \int x|v_1|^2dx - 3 \int |v_{1,x}|^2 dx = \int g_1v_1 dx.
\end{equation*}
Cauchy-Schwarz inequality implies
\begin{equation*}
    R\| v_1\|^2_{L^2}+ 3\|v_{1,x}\|^2_{L^2}\lesssim \|g_1\|_{L^2}\|v_1\|_{L^2}\leq \frac{1}{2R}\|g_1\|_{L^2}^2 + \frac{R}{2}\| v_1\|^2_{L^2},
\end{equation*}
which further leads to
\begin{equation*}
   R\| v_1\|^2_{L^2}+ \|v_{1,x}\|^2_{L^2}\lesssim R^{-1}\|g\|^2_{L^2}.
\end{equation*}
Thus, using the bound \eqref{FOneBound}, we arrive at
\begin{equation*}
    \|v_1\|_{L^2}\lesssim \epsilon R^{-\frac{5}{4}}, \quad \|\partial_x v_1\|_{L^2}\lesssim \epsilon R^{-\frac{3}{4}},
\end{equation*}
and further using the $v_1$ equation,
\begin{equation*}
    \|\partial_x^2 v_1\|_{L^2}\lesssim \epsilon R^{-\frac{1}{4}}.
\end{equation*}
Now we can obtain pointwise bounds for $v_1$ by interpolation and Sobolev embeddings 
\begin{equation*}
    |v_1|\lesssim \epsilon R^{-1}, \quad |\partial_x v_1|\lesssim \epsilon R^{-\frac{1}{2}}.
\end{equation*}
This is exactly as needed. 
On the other hand for the $x^{-1}g_{lo}$ we proceed as 
before, using Bernstein’s inequality, in order to obtain a similar bound but with a logarithm loss.

\bibliographystyle{plain}
\bibliography{refs}
\end{document}